\documentclass{isypaper}

\usepackage{breqn}
\usepackage{cite}
\usepackage[english]{babel}
\usepackage{booktabs}
\usepackage{multirow}

\begin{document}

\title{Isometries of Length $1$ in Purely Loxodromic Free Kleinian Groups and Trace Inequalities}  
\author{A. Nedim Narman and \.Ilker S. Y\"uce}

%\address{T. C. Yeditepe University\\ Faculty of Art and Sciences\\ Department of Mathematics}

\newcommand{\hyp}{\mathbb{H}^3}

%Nedim defined commands
\newcommand{\Ra}{\Rightarrow}
\newcommand{\ra}{\rightarrow}
\newcommand{\xii}{\xi_i}
\newcommand{\xij}{\xi_j}
\newcommand{\xik}{\xi_k}
\newcommand{\C}{\mathbb{C}}
\newcommand{\simp}{\bigtriangleup}
\newcommand{\raa}[1]{\renewcommand{\arraystretch}{#1}}
\setlength\tabcolsep{8pt}
\newcommand{\xigen}[4]{\xi_{#1}^{#2}\xi_{#3}^{#4}}
\newcommand{\xigent}[6]{\xi_{#1}^{#2}\xi_{#3}^{#4}\xi_{#5}^{#6}}
\newcommand{\set}[2]{\{{#1}\ |\ {#2}\}}
\newcommand{\F}{\mathcal{F}}
\newcommand{\G}{\mathcal{G}}
\newcommand{\Pl}{\mathcal{P}}
\newcommand{\psl}{\textnormal{PSL}(2,\mathbb{C})}
\newcommand{\gfFrak}{\mathfrak{G}\mathfrak{F}}
\newcommand{\trace}[1]{\textnormal{trace}(#1)}
\newcommand{\tracesquare}[1]{\textnormal{trace}^2(#1)}

\numberwithin{equation}{section}

\maketitle

\begin{abstract}

\noindent In this paper, we prove a generalization of a discreteness criteria for a large class of subgroups of PSL$_2(\mathbb{C})$. In particular, we show that for a given finitely generated, purely loxodromic, free Kleinian group $\Gamma=\langle\xi_1,\xi_2,\dots,\xi_n\rangle$ for $n\geq 2$, the inequality 
	$$\left|\text{trace}^2(\xi_i)-4\right|+\left|\text{trace}(\xi_i\xi_j\xi_i^{-1}\xi_j^{-1})-2\right|\geq 2\sinh^2\left(\frac{1}{4}\log\alpha_n\right)$$
	holds for some $\xi_i$ and $\xi_j$ for $i\neq j$ in $\Gamma$ provided that certain conditions on the hyperbolic displacements given by $\xi_i$, $\xi_j$ and their length $3$ conjugates formed by the generators are satisfied. Above, the constant $\alpha_n$ turns out to be the real root strictly larger than $(2n-1)^2$ of a fourth degree, integer coefficient polynomial obtained by solving a family of optimization problems via Karush-Kuhn-Tucker theory. The  use of this theory in the context of hyperbolic geometry is another novelty of this work.  
\end{abstract}

\section{Introduction}\label{Intro}

By the work of Thurston \cite{WPT1, WPT2}, Jorgensen and Gromov \cite{MG}, it is known that the volume is a geometric invariant of finite-volume hyperbolic $3$-manifolds. Furthermore, by the Mostow-Prasad Rigidity Theorem (see  \cite{BP,GDM}) it is also a topological invariant. Therefore, investigation of the connections between the volume and the usual invariants of topology is one of the main topics of research in the study of finite-volume hyperbolic $3$-manifolds. 

A particular approach practiced in this area of research is to explore the implications of discreteness criteria such as the Jorgensen's inequality \cite{JT} which states that 
$|\textnormal{trace}^2(\xi_1)-4|+|\textnormal{trace}(\xi_1\xi_2\xi_1^{-1}\xi_2^{-1})-2|\geq 1$ if $\Gamma=\langle\xi_1,\xi_2\rangle$ is a Kleinian group. Trace inequalities of this sort  can be used not only to calculate estimates for the volumes but also estimates for geometric quantities like the injectivity radius, diameter, etc.  A seminal result along this line due to Meyerhoff  \cite{RGM} which states that  $0.104$ is a Margulis constant for $n=3$.  In other words, for isometries $\xi$ and $\eta$ of $\hyp$ generating a non-elementary discrete group the inequality $\max\{d_{\xi}z,d_{\eta}z\}\geq 0.104$ holds for any $z\in\hyp$, where $d_{\gamma}z$ denotes the hyperbolic distance between $z\in\hyp$ and $\gamma(z)$ for an isometry $\gamma$ of $\hyp$. The existence of such constants is implied by the Margulis Lemma \cite{BP}. They have implications on the volumes of closed hyperbolic $3$-manifolds \cite{RGM}. 

An analogous lower bound independent from the Jorgensen's inequality for the maximum of the hyperbolic displacements by $\xi$ and $\eta$ known as the $\log 3$ Theorem is proved by Culler and Shalen in \cite{CSParadox}. This theorem states that $\max\{d_{\xi}z,d_{\eta}z\}\geq\log 3$ holds for any $z\in\hyp$ provided that $\xi$ and $\eta$ generate a purely loxodromic free Kleinian group  \cite{CSParadox, PBS}. A brief survey by Shalen on the implications of the $\log 3$ theorem and its generalization called the $\log(2k-1)$ Theorem \cite{ACCS}  on the volume of the closed hyperbolic $3$-manifolds can be found in \cite{PBS}.

The lower bound given in the Jorgensen's inequality \cite{JT}  is the best possible. It is conceivable that if one considers more restrictive Kleinian groups such as the classes of purely loxodromic and free Kleinian groups as a particular example the class of all finitely generated Schottky groups \cite{BM}, a larger universal lower bound in the Jorgensen's inequality for these classes may be achieved. Consequently, the implications of this lower bound on the geometric and topological invariants of finite-volume hyperbolic $3$-manifolds can be investigated for finite-volume hyperbolic $3$-manifolds whose fundamental groups are $k$-free. A group $\Gamma$ is called $k$-free if every finitely generated subgroup of $\Gamma$ of rank at most $k$ is free. With this motivation, the following refinement of the Jorgensen's Inequality is proved in \cite{Yu4}:
\begin{theorem}\label{thm1.3}
Let $\Gamma=\langle\xi_1,\xi_2\rangle$ be a purely loxodromic free Kleinian group. If the following inequalitıes hold
	\begin{itemize}
	\item[(i)]  $d_{\gamma}z_2<1.6068...$ for every $\gamma\in\{\xi_2,  \xi_1^{-1}\xi_2\xi_1, \xi_1\xi_2\xi_1^{-1}\}$ for the midpoint $z_2$ of the shortest geodesic segment connecting the axis of $\xi_1$ to the axis of  $\xi_2^{-1}\xi_1\xi_2$ and
	\item[(ii)]	$d_{\xi_2\xi_1\xi_2^{-1}}z_2\leq d_{\xi_2\xi_1\xi_2^{-1}}z_1$ for the midpoint $z_1$ of the shortest geodesic segment connecting the axis of $\xi_1$ to the axis of  $\xi_2\xi_1\xi_2^{-1}$
	\end{itemize} 
	 then, we have $|\textnormal{trace}^2(\xi_1)-4|+|\textnormal{trace}(\xi_1\xi_2\xi_1^{-1}\xi_2^{-1})-2| \geq 1.5937....$
\end{theorem}

In this paper, we shall prove the statement below proposed in \cite{Yu4} as a generalization of \fullref{thm1.3}:
\begin{theorem}\label{thm1.4}
	Let $\Gamma=\langle\xi_1,\xi_2\dots,\xi_n\rangle$ be a purely loxodromic free Kleinian group. If there exist generators $\xi_i$ and $\xi_j$ for $i\neq j$ such that 
	\begin{itemize}
		\item[(i)]  $d_{\gamma}z_2<({1}/{2})\log\alpha_n$ for every $\gamma$ a generator or a length $3$ conjugate formed by the generators other than $\xi_i$, $\xi_j^{-1}\xi_i\xi_j$ and $\xi_j\xi_i\xi_j^{-1}$ for the midpoint $z_2$ of the shortest geodesic segment connecting the axis of $\xi_i$ to the axis of  $\xi_j^{-1}\xi_i\xi_j$ and
		\item[(ii)] $d_{\xi_j\xi_i\xi_j^{-1}}z_2\leq d_{\xi_j\xi_i\xi_j^{-1}}z_1$ for the midpoint $z_1$ of the shortest geodesic segment connecting the axis of $\xi_i$ to the axis of  $\xi_j\xi_i\xi_j^{-1}$, then the inequality
	\end{itemize} 
	\begin{equation}\label{GJ}
		|\textnormal{trace}^2(\xi_i)-4|+|\textnormal{trace}(\xi_i\xi_j\xi_i^{-1}\xi_j^{-1})-2|\geq 2\sinh^2\left(\tfrac{1}{4}\log\alpha_n\right) 
	\end{equation}
holds. Above, $\alpha_n$ is the only root  greater than $(2n-1)^2$ of the polynomial
	\begin{equation}\label{P}
		\begin{multlined}
			\mathcal{P}(\lambda)=(8n^3-12n^2+2n+1)\ \lambda^4+\\
			\shoveleft[2cm]{(-64n^6+192n^5-192n^4+64n^3+4n^2+2n-4)\ \lambda^3}\ +\\
			\shoveleft[3.5cm]{(-96n^5+224n^4-168n^3+52n^2-18n+6)\ \lambda^2}\ +\\
			\shoveleft[5cm]{(32n^5-112n^4+128n^3-68n^2+22n-4)\ \lambda}\ +\\
			\shoveleft[8.8cm]{16 n^4-32 n^3+24 n^2-8 n+1}.
		\end{multlined}
	\end{equation}
\end{theorem}

The proof of \fullref{thm1.3} uses a result of Beardon \cite[Theorem 5.4.5]{Beardon} which connects upper bounds for trace inequalities to a lower bound for the maximum of  hyperbolic displacements by loxodromic isometries of $\hyp$ and techniques of calculating lower bounds introduced in \cite{Yu,Yu2} for the maximum of hyperbolic displacements by loxodromics.  The proof of \fullref{thm1.4} requires the use of the same tools with one major challenge due to 
the generalization of \fullref{thm1.3} to an arbitrary number of generators. This makes the necessary calculations for a lower bound for the maximum of hyperbolic displacements by loxodromics much more complicated. A direct analogue of the paper \cite{Yu4} for even $n=3$ would have been technically very cumbersome. To overcome this difficulty, a new notational system is introduced in \fullref{DefNot}.  With this notational system, a simplified version of \fullref{thm1.4} is rephrased as \fullref{thened:main_result} in \fullref{sec:theend}. 

A brief summary of the proof of \fullref{thm1.4} and the organization of this paper can be given as follows: main step in the proof of \fullref{thm1.4} 
is the calculation of a lower bound for the maximum of the hyperbolic displacements given by the generators of $\Gamma=\langle\xi_1,\xi_2,...,\xi_n\rangle$ and their length $3$ conjugates. This is achieved in \fullref{theend:log_alpha} in \fullref{sec:theend} by considering two cases: $\Gamma$ is geometrically infinite or $\Gamma$ is geometrically finite. The machinery introduced by Culler and Shalen to prove the $\log 3$ Theorem \cite{CSParadox} is used in both cases. The methods developed in \cite{Yu,Yu2} are needed in the first case. 

Assuming $\Gamma=\langle\xi_1,\xi_2,...,\xi_n\rangle$ is geometrically infinite, a carefully chosen decomposition denoted by $\Gamma_{\mathcal{D}}$ of $\Gamma$ concentrating on the length $3$ conjugates formed by the generators is defined in \fullref{DefNot}. This decomposition leads to a decomposition of the Patterson-Sullivan measure which is the area measure on the limit set of $\Gamma$ homeomorphic to the sphere at infinity. The group-theoretical relations of $\Gamma_{\mathcal{D}}$ provide some measure-theoretical relations given in  \fullref{CSMachine:measures}. With a key lemma \cite[Lemma 5.5]{CSParadox}, these measure-theoretical relations provide lower bounds for the hyperbolic displacements by the isometries determined by the decomposition $
\Gamma_{\mathcal{D}}$ of $\Gamma$. The lower bounds are considered as the values of a certain set of functions $\{f_r\}$, referred to as the displacement functions, evaluated at a point $m$ in a certain simplex $\Delta$. This is phrased in \fullref{CSmachine:main}. The infimums $\alpha_n$ of the maximum of the displacement functions $\{f_r\}$ over $\Delta$ provide the lower bounds $(1/2)\log\alpha_n$ for the maximum of the hyperbolic displacements by the generators of $\Gamma$ and their length $3$ conjugates. In  \fullref{CSMachine}, we summarize the relevant parts of the Culler-Shalen machinery for this case of the proof, particularly, by underlining the importance and applications of the group theoretical relations. 

All of \fullref{sec:Optimization} is devoted to the calculation of the infimums $\alpha_n$, obtained by solving a family of nonlinear optimization problems via the Karush-Kuhn-Tucker Theory \cite{convex}, completing the proof of \fullref{theend:log_alpha} when $\Gamma$ is geometrically infinite. The increase in the number of displacement functions $\{f_r\}$ due to the increase in the number of generators is tackled by the use of inherent symmetries of the decomposition $\Gamma_{\mathcal{D}}$ of $\Gamma$ from the start rather than towards the end as it was done in \cite{Yu4}. This simplifies the entire process as the symmetries are not encrypted into the displacement functions. The use of these symmetries is achieved via classifications of the decomposition elements and group theoretical relations as they are discussed and presented in \fullref{sec:Tables} with tables \fullref{tab:1} and \fullref{tab:2}. On a side note, due to the effectiveness of these classifications, \fullref{sec:Optimization} is largely a simplified and generalized version of the calculations in \cite{Yu4}. Almost all the ideas that are used in this section are from \cite{Yu4} as they were presented in this paper but, they are translated to the language of the classifications in \fullref{tab:1} and \fullref{tab:2} in \fullref{sec:Tables}.

The second case, $\Gamma=\langle\xi_1,\xi_2,...,\xi_n\rangle$ is geometrically finite, in the proof of \fullref{theend:log_alpha} is considered in Section \ref{sec:theend}. This case can be reduced to geometrically infinite case by the use of a number of deep results \cite[Propositions 8.2 and 9.3]{CSParadox}, \cite[Main Theorem]{CSH} and \cite{CCHS} from the deformation theory of Kleinian groups completing the proof of this theorem. Once \fullref{theend:log_alpha} is at hand, the main result \fullref{thm1.4} (\fullref{thened:main_result} in \fullref{sec:theend}) follows from a proof by contradiction.  

Under the hypothesis of \fullref{thm1.4}, assuming the strictly less than inequality in (\ref{GJ}) implies the existence of a point in $\hyp$ so that the infimum of the maximum of the hyperbolic displacements given by the generators and their length $3$ conjugates in $\Gamma=\langle\xi_1,\xi_2,...,\xi_n\rangle$ at that point becomes strictly less than $(1/2)\log\alpha_n$ contradicting with \fullref{theend:log_alpha}. This is done by reversing the inequalities in the proof of a theorem by Beardon \cite[Theorem 5.4.5 (iii)]{Beardon} which states that if $\xi_1$ is elliptic or strictly loxodromic and $|\textnormal{trace}(\xi_1)-4|<1/4$, then 
\begin{equation*}
\max\left\{\sinh\frac{1}{2}d_{\xi_1}z,\sinh\frac{1}{2}d_{\xi_2\xi_1\xi_2^{-1}}z\right\}\geq\frac{1}{4}
\end{equation*}
for all $z\in\hyp$ provided that $\langle\xi_1,\xi_2\rangle$ is discrete and non-elementary. A number of auxiliary lemmas needed in the proof of \fullref{thened:main_result} are also proved in the final section.

\section{Basic Definitions and Notations}\label{DefNot}

Throughout this text, we will work with a fixed subset $\Xi=\{\xi_1,\ \xi_2,\ \dots,\ \xi_n\}$ of $PSL(2,\C)$ such that $\Gamma=\langle \Xi \rangle$ is a purely loxodromic group freely generated by $\Xi$.
Let $\Xi^{-1}=\{\xi_1^{-1},\ \xi_2^{-1},\ \dots,\ \xi_n^{-1}\}$. The sets of particular interest  in this paper will be the subsets $\Gamma_1$, $\Psi$ and $\Gamma_*$ of $\Gamma$ defined as %$\Gamma_1 =\{1\}\cup\Xi\cup\Xi^{-1}$, 
\begin{eqnarray*}
    \Psi     & = & \left\{\xii^{2t},\ \xii^t\xij^{2s},\ \xii^t\xij^s\xik^p \  |\  i,j,k\in\{1, 2, \dots , n\},\ i\neq j,\ j\neq k,\ t,s,p\in\{-1,+1\}\right\},\\
    \Gamma_* & = & \left\{1,\ \xii^t,\ \xii^t\xij^s\xii^{-t},\ |\  i,j\in\{1, 2, \dots , n\},\ i\neq j,\ t,s\in\{-1,+1\}\right\},\\
    \Gamma_1 & = & \{1\}\cup\Xi\cup\Xi^{-1}.
\end{eqnarray*}
In general, we will consider $i,j,k,t,s,p$ varying in their respective ranges arbitrarily and use sub-indexes such as $i_0, j_0, i_1$ to represent their values once they are fixed.
We will also refrain from stating obvious equalities or inequalities such as $i\neq j$ and $t=\pm 1$. This notational compromise leads to the following relaxations of the definitions above 
\begin{equation*}
\Psi =\left\{\xii^{2t},\ \ \xii^t\xij^{2s},\ \ \xii^t\xij^s\xik^p\right\}\ \textnormal{ and }\ \Gamma_*  =\left\{1,\ \xii^t,\ \xii^t\xij^s\xii^{-t}\right\}.
\end{equation*}

For $\psi\in\Psi$, let $J_\psi$ represents the elements of $\Gamma$ that start with $\psi$ in their simplified form. Notice that $(J_\psi)_{\psi\in\Psi}$ gives a collection of disjoint subsets of $\Gamma$. Moreover, we have the following partition of $\Gamma$:
\begin{equation}\label{decomp}
    \Gamma=\{1\}\sqcup\Xi\sqcup\Xi^{-1}\sqcup\bigsqcup_{\psi\in\Psi}J_\psi.
\end{equation}
The setting $\mathcal{D}=(\Psi, \Gamma_*)$ together with the partition above is said to be a decomposition $\Gamma_{\mathcal{D}}$ of $\Gamma$.
A group theoretical relation $r$ of the decomposition $\Gamma_{\mathcal{D}}$ is defined to be a triple $(\gamma_r,\ \psi_r,\ \Psi_r)$ satisfying the followings:
\begin{eqnarray}\label{relations}
r\ :\ \gamma_r\in\Gamma_*,\ \ \psi_r\in\Psi,\ \ \Psi_r\subseteq\Psi & \textnormal{and} & \gamma_r J_{\psi_r}=\Gamma - \left(\{\cdot\}\cup J_{\Psi_r}\right).
\end{eqnarray}
To better understand this definition, consider $r$ with $\gamma_r=\xi_{1}^{-1}$ and $\psi_r=\xi_{1}\xi_{2}\xi_{3}$.
It is easy to see that $\gamma_r J_{\psi_r}$ will consist of all the elements of $\Gamma$ starting with $\xi_2\xi_3$. The term $(\{\cdot\}\cup J_{\Psi_r})$ in the relation equation should be all elements not starting with $\xi_2\xi_3$.
If we take $\Psi_r=\Psi-\{\xi_2\xi_3,\ \xi_2\xi_3\xi_i^t\}$, then $J_{\Psi_r}=\bigcup_{\psi\in\Psi_r}J_{\psi}$ will exactly be what needs to be removed from $\Gamma$.
Hence, the followings $$r: \gamma_r=\xi_{1}^{-1},\ \psi_r=\xi_{1}\xi_{2}\xi_{3},\ \Psi_r=\Psi-\left\{\xi_2\xi_3^2,\xi_2\xi_3\xii^t\right\},\ \{\cdot\}=\Gamma_1$$
define a group theoretical relation. Including the relation above, all relations are quite symmetrical. We can generalize them into families of relations as follows:
\begin{equation*}
r: \gamma_r=\xi_{i_0}^{-t_0},\ \psi_r=\xi_{i_0}^{t_0}\xi_{j_0}^{s_0}\xi_{k_0}^{p_0},\ \Psi_r=\Psi-\left\{\xi_{j_0}^{s_0}\xi_{k_0}^{2p_0},\ \xi_{j_0}^{s_0}\xi_{k_0}^{p_0}\xii^t\right\},\ \{\cdot\}=\Gamma_1.
\end{equation*}
Next, we shall simplify the notation for displacement functions. Let $x\in\mathbb{R}^\Psi$ and $r$ be a relation. We define
\begin{equation*}
x_r=x(\psi_r),\ X_r=\sum_{\psi\in\Psi_r}x(\psi) \text{ and } X=\sum_{\psi\in\Psi}x(\psi).
\end{equation*}
Although this notation has been described above only for $x$, we shall use capital letters almost exclusively for this purpose.
Accordingly, the formula of the displacement function $f_r$ of the relation $r$ is defined as:
\begin{eqnarray*}
    f_r:\simp \longrightarrow \mathbb{R}_{++}   & \textnormal{such that} &  
   f_r( x ) = \frac{1-x_r}{x_r}\cdot\frac{1-X_r}{X_r},
\end{eqnarray*}
where $\simp$ will denote the simplex $\{x\in\mathbb{R}^\Psi_{++} \ |\ X=1\}$.  To define the $\alpha$ used in the rest of the paper, we use the polynomial $\mathcal{P}(\lambda)$ in (\ref{P})
and introduce the  lemma below:
\begin{lemma}
    \label{Proots}
The polynomial $\mathcal{P}(\lambda)$ has a unique root larger than $(2n-1)^2$ and all other roots are less than $1$.
\end{lemma}
\begin{proof}
All four roots of $\mathcal{P}(\lambda)$ are real and distinct for $n\geq 2$. Therefore, it suffices to find certain intervals such that the evaluation on either side have opposite signs.
The first 4 rows of \fullref{tab:proots} below proves that $\mathcal{P}(\lambda)$ has three roots between $-2$ and $1$. 
\begin{table}[H]
        \centering
        \raa{1.5}
        \begin{tabular}{cclc}
            \toprule
            $\lambda$                         &   & $\Pl (\lambda)$                              & sign of $\Pl$ \\
            \midrule
            $-2$                              &   &
            $(2n-3)(256n^5-608n^4 +424n^3-36n^2+18n-17)$            &
            pos.                                                                                                                   \\
   %                                           &   & \hspace*{2cm}$+424n^3-36n^2+18n-17)$                                                \\
            \midrule
            $\displaystyle\frac{-1}{n}$                            &   &
            $-(16n^8-48n^7+72n^6-84n^5+33n^4+10n^3+8n^2-6n-1)/n^4$       &
            neg.                                                                                                                   \\
  %                                            &   & \hspace*{2cm}$+33n^4+10n^3+8n^2-6n-1)/n^4$                                          \\
            \midrule
            $\displaystyle\frac{-1}{2n-1}$                       &   &
            $32n^4(n-1)/(2n-1)^2$             &
            pos.                                                                                                                   \\
            \midrule
            $1$                               &   &
            $-64n^4(n-1)^2$                   &
            neg.                                                                                                                    \\
            \midrule
            $(2n-1)^2$                        &   &
            $-128n^4(2n^2-3n+1)^3(4n^2-8n+5)$ &
            neg.                                                                                                                    \\
            \midrule
            $(2n-1)^3$                        &   &
            $16(2n-1)^4(n-1)^2(32n^7-80n^6+56n^5+4n^4-22n^3+16n^2-5n+1)$   &
            pos.                                                                                                                   \\
%                                            &   & \hspace*{1cm}$+56n^5+4n^4-22n^3+16n^2-5n+1)$                                        \\
            \bottomrule
        \end{tabular}
        \caption{Roots of $\mathcal{P}$.}
	   \label{tab:proots}
    \end{table}
\noindent The last two rows prove that $\mathcal{P}(\lambda)$ has a root larger than $(2n-1)^2$. Hence, we located all four roots.
\end{proof}
\noindent By \fullref{Proots}, we can formally give the definition of $\alpha=\alpha_n$: the unique root of $\mathcal{P}(\lambda)$ larger than $(2n-1)^2$.

%%%%%%%%%%%%%%%%%%%%%%%%%%%%%%%%%%%% SECTION 3

\section{The Culler-Shalen Machinery}\label{CSMachine}

In the rest of the paper, we will use $S_\infty$ to denote the boundary of the canonical compactification of $\mathbb{H}^3$.
The decompositions $\Gamma_{\mathcal{D}}$ of $\Gamma$ in (\ref{decomp}) and their group theoretical relations defined in (\ref{relations}) in the previous section allow one to decompose the area measure of $S_\infty$ into certain Borel measures and  obtain some measure theoretical relations among these Borel measures, respectively.  Consequently, measure theoretical relations lead to the lower bounds for the hyperbolic displacements by the isometries used in the definition of $\Gamma_{\mathcal{D}}$.
This will be done in an almost identical way in \cite[Lemma 5.3]{CSParadox},  \cite[Lemma 3.3, Theorem 3.4]{Yu} and \cite[Theorem 2.1]{Yu2}. Therefore, we phrase the basic theorems below without any proofs:
\begin{theorem}\label{CSMachine:measures}
	Let $z\in \mathbb{H}^3$ and $A_z$ be the area measure of $S_\infty$ at $z$. If $\Gamma$ is geometrically infinite, then there exists a family of Borel measures $\{\nu_\psi\}_{\psi\in\Psi}$ on $S_\infty$ such that:
		\begin{equation*}
		(i)\ A_z=\sum_{\psi\in\Psi}\nu_\psi,\quad
		(ii)\ A_z(S_\infty)=1,\quad
		(iii)\  \int_{S_\infty}\lambda_{\gamma_r,z}^2 d\nu_{\psi_r}=1-\sum\limits_{\psi\in\Psi_r}\nu_{\psi}(S_\infty)\ 
		%for\ any\ relation\ r.
		\end{equation*}
for any relation $r$.
\end{theorem}
With \fullref{CSMachine:measures} , we can define a special element 
	$m(\psi)=v_\psi(S_\infty)\in\mathbb{R}^\Psi $
for all  $\psi\in\Psi$ for any given $z\in \mathbb{H}^3$. We take $\{\nu_\psi\}_{\psi\in\Psi}$ as in \fullref{CSMachine:measures}. We also use the notations of \cite{Beardon} for the hyperbolic displacement of $z$.
If $\gamma$ is an isometry of $\mathbb{H}^3$ with respect to the hyperbolic metric $\rho$, then
	$d_\gamma z=\rho(z,\gamma z).$

\newpage

The crucial part of the Culler-Shalen machinery for this paper is that it gives lower bounds for the hyperbolic displacements given by the isometries.
This is due to \cite[Lemma 5.5]{CSParadox} and its improved version \cite[Lemma 2.1]{CSMargulis}. We present the latter:

\begin{lemma} \label{CSMachine:dispIneq}
	Let $a$ and $b$ be numbers in $[0,1]$ which are not both equal to $0$ and are not both equal to $1$.
	Let $\gamma$ be a loxodromic isometry of $\mathbb{H}^3$ and let $z\in\mathbb{H}^3$. Suppose that $\nu$ is a measure on $S_\infty$ such that
	\[v\leq A_z,\ v(S_\infty)\leq a \text{ and } \int_{S_\infty}\lambda_{\gamma,z}^2 d\nu \geq b.\]
	Then, we have
	\[a>0,\ b<1 \text{ and }\ d_\gamma z\geq \frac{1}{2}\log\frac{b(1-a)}{a(1-b)}.\]
\end{lemma}

Finally, we merge the results of \fullref{CSMachine:measures} and \fullref{CSMachine:dispIneq}  to obtain the lower bounds for the hyperbolic displacements that we will need in \fullref{sec:theend}.

\begin{proposition}\label{CSmachine:main}
	Let $z\in\mathbb{H}^3$ and $m$ be the relevant vector in $\mathbb{R}^\Psi$.
	If $\Gamma$ is geometrically infinite, then for any $\gamma\in\Gamma_*$ and for any relation $r$ defined by $\Psi$ we have:
	\begin{equation*}
	(i)\ m\in\simp, \quad 
	(ii)\ e^{2d_{\gamma} z}\geq f_r(m)\ for\ any\ relation\ r\ with\ \gamma_r=\gamma.
	\end{equation*}
\end{proposition}
\begin{proof}
	First, we shall prove that $m(\psi)\neq 0$ for all $\psi\in\Psi$.
	We do so by starting with the conjugate elements of $\Psi$ and then generalizing to the rest of the elements.
	
	Assume $m(\xigent{i_0}{t_0}{j_0}{s_0}{i_0}{-t_0})=0$.
	Let $r_0$ be the relation defined by
	$r_0: \gamma_0=\xigent{i_0}{t_0}{j_0}{-s_0}{i_0}{-t_0}$, $\psi_0=\xigent{i_0}{t_0}{j_0}{s_0}{i_0}{-t_0},$ $\Psi_0=\{\xi_{i_0}^{2t_0},\ \xigen{i_0}{t_0}{j}{2s},$ $\xigent{i_0}{t_0}{j}{s}{k}{p}\}.$
	Since $m(\psi_0)=0$, we have $M_0=\sum_{\psi\in\Psi_0}m(\psi)=1$ by \fullref{CSMachine:measures}(\textit{iii}).
	This implies that there exists a $\psi_1\in\Psi_0$ such that $m(\psi_1)\neq 0$. Take $\psi_2=\xigent{i_2}{t_2}{j_2}{s_2}{i_2}{-t_2}$, where $i_2 \neq i_0$ so that $\psi_2 \notin \Psi_0$.
	Notice that $M=\sum_{\psi\in\Psi}m(\psi)=1$ by \fullref{CSMachine:measures} (\textit{i}) and (\textit{ii}).
	Now, as both $M$ and $M_0$ are equal to $1$, we have $\sum_{\psi \in \Psi-\Psi_0}m(\psi)=0$.
	This implies that $m(\psi_2)=0$.

	Let $r_2$ be the relation defined by:
	$r_2: \gamma_2=\psi_2^{-1},$ $\psi_2=\psi_2,$ $ \Psi_2=\{\xi_{i_2}^{2t_2}, \xigen{i_2}{t_2}{j}{2s}, \xigent{i_2}{t_2}{j}{s}{k}{p}\}$
	so that $M_2=1$ and $M-M_2=0$. This gives that $m(\psi)=0$ for all $\psi \in \Psi-\Psi_2$,
	which  contradicts with the existence of $\psi_1$ as $\Psi_0 \subset \Psi-\Psi_2$.
	Hence, we have the following fact
	\begin{align}\label{type4mnot0}
		m(\xigent{i}{t}{j}{s}{i}{-t}) \neq 0.
	\end{align}
	Next, assume that $m(\xi_{i_0}^{2t_0})=0$. Take $r_0$ to be the group theoretical relation $r_0:\ \gamma_0=\xi_{i_0}^{-t_0},$ $\psi_0=\xi_{i_0}^{2t_0},$ $\Psi_0=\Psi-\{\xi_{i_0}^{2t_0},\ \xigen{i_0}{t_0}{j}{2s},\ \xigent{i_0}{t_0}{j}{s}{k}{p}\}.$
	By \fullref{CSMachine:measures}, we obtain $m(\psi)=0$ for all $\psi\in\Psi-\Psi_0$.
	Thus, $m(\xigent{i_0}{t_0}{j_0}{s_0}{i_0}{-t_0})=0$ contradicting (\ref{type4mnot0}).
	This proves that $m(\xi_{i}^{2t})\neq 0.$

Using similar arguments given above, we proceed as follows: 
assume that $m(\xigen{i_0}{t_0}{j_0}{2s_0})=0$. We consider the relation  $r_0:\ \gamma_0=\xigent{i_0}{t_0}{j_0}{-s_0}{i_0}{-t_0},$ $\psi_0=\xigen{i_0}{t_0}{j_0}{2s_0},$ $\Psi_0=\Psi-\{\xigen{i_0}{t_0}{j_0}{2s_0},\ \xigent{i_0}{t_0}{j_0}{s_0}{k}{p}\}$. Then, we derive that $m(\xigent{i_0}{t_0}{j_0}{s_0}{i_0}{-t_0})=0$ which contradicts with (\ref{type4mnot0}).
			%\midrule
Assume that $m(\xigent{i_0}{t_0}{j_0}{s_0}{i_0}{t_0})=0$. Consider the relation                               
			$r_0:\ \gamma_0=\xi_{i_0}^{-t_0},$ $\psi_0=\xigent{i_0}{t_0}{j_0}{s_0}{l_0}{p_0},$ $\Psi_0=\Psi-\{\xigen{j_0}{s_0}{i_0}{2t_0},\ \xigent{j_0}{s_0}{i_0}{t_0}{k}{p}\}$. Then, we get  $m(\xigent{j_0}{s_0}{i_0}{t_0}{j_0}{-s_0})=0$, a contradiction by  (\ref{type4mnot0}).  Finally, assume that $m(\xigent{i_0}{t_0}{j_0}{s_0}{k_0}{p_0})=0$. We consider the relation $r_0:\ \gamma_0=\xi_{i_0}^{-t_0},$ $\psi_0=\xigent{i_0}{t_0}{j_0}{s_0}{k_0}{p_0},$ $\Psi_0=\Psi-\{\xigen{j_0}{s_0}{k_0}{2p_0},\ \xigent{j_0}{s_0}{k_0}{p_0}{k}{p}\}$. Then, we find that  $m(\xigent{i_0}{t_0}{j_0}{s_0}{i_0}{-t_0})=0$, again a contradiction by  (\ref{type4mnot0}). Therefore, we derive that 
	$m(\xigen{i}{t}{j}{2s})\neq 0,$ $m(\xigent{i}{t}{j}{s}{i}{t})\neq 0,$ $\textnormal{and}\ m(\xigent{i}{t}{j}{s}{k}{p})\neq 0.$
	We cycled over all element types in $\Psi$ showing $m\in\mathbb{R}_{++}^\Psi$.
	Since we already know that $M=1$, this proves ($i$) of the current proposition.
	
	To prove ($ii$), fix any relation $r$ with $\gamma_r=\gamma$. Take $a=\nu_{\psi_r}(S_\infty)=m(\psi_r)$ and $b=\int_{S_\infty}\lambda^2_{\gamma,z}d\nu_{\psi_r}$. Notice that $a$ and $b$ satisfy the conditions of \fullref{CSMachine:dispIneq}. Thus, we obtain
	\[d_{\gamma} z\geq \frac{1}{2}\log\left(\frac{1-a}{a}\cdot \frac{b}{1-b}\right).\]
	Moreover, if $a=m_r$ and  $b=1-\sum_{\psi\in\Psi_r}m(\psi)=1-M_r$ (by \fullref{CSMachine:measures}($iii$)), then the inequality above  turns into
	\[d_{\gamma} z\geq \frac{1}{2}\log\left( \frac{1-m_r}{m_r}\cdot \frac{1-M_r}{M_r}\right)=\frac{1}{2}\log f_r(m).\]
	Taking the exponential of the last expression yields ($ii$).
\end{proof}

The proposition above shows that we can use the displacement functions of the relations $r$ and $m$ to find lower bounds for hyperbolic displacements.
We could sharpen this bound by searching for the maximal value of $f_r(m)$ over a family of relations.
Obviously, the larger the family, the better the lower bound for $\text{max}\set{f_r(m)}{r}$. Unfortunately, calculating values for $m$ for every $z\in\mathbb{H}^3$ is not feasible.
Therefore, we look at all possible values of $m$ within the bounds set by \fullref{CSmachine:main}($i$).
As we don't know which element of $\simp$ is actually $m$, we need to calculate the infimum over all elements.
Ultimately, we shall prove in \fullref{sec:theend} that what matters is the infimum of the max function of a collection of displacement functions over $\simp$. This optimization is performed in \fullref{sec:Optimization}.

%%%%%%%%%%%%%%%%%%%%%%%%%%%%%%%%%%%%% SECTION 4

\section{The Max Functions and Classification of Relations}\label{sec:Tables}

In this section, we mainly present some tables illuminating the mechanics of the displacement functions for the relevant decompositions $\Gamma_{\mathcal{D}}$.
We also discuss how these tables are used in later sections and the objectives behind their discovery. \fullref{tab:1} below is used to count and classify elements of $\Psi$ into 5 types.
\begin{table}[H]
	\centering
	\raa{1.5}
	\begin{tabular}{ccr}
		\toprule
		& $\psi$                    & count                  \\
		\midrule
		type 1   & $\xii^{2t}$               & $2n$                   \\
		type 2   & $\xii^t\xij^{2s}$         & $4n(n-1)$              \\
		type 3   & $\xii^t\xij^{s}\xii^t$    & $4n(n-1)$              \\
		type 4   & $\xii^t\xij^{s}\xii^{-t}$ & $4n(n-1)$              \\
		type 5   & $\xii^t\xij^{s}\xik^p$    & $8n(n-1)(n-2)$         \\
		\midrule
		$|\Psi|$ & $|\{i,\ j,\ k\}|=3$       & $2n+4n(n-1)+8n(n-1)^2$ \\
		\bottomrule
	\end{tabular}
	\caption{Types of elements in $\Psi$.}
	\label{tab:1}
\end{table}

While reading this table, remember the notational relaxations that are discussed in \fullref{DefNot} and also assume $|\{i,\ j,\ k\}|=3$ at any instance.
The above classification (with the counts) will be useful in simplifying the calculations for the optimization question which is discussed in \fullref{Intro}.
However, what we really need is a set of coherent relations that will yield many symmetries of the simplex $\simp$.

Notice that for the purposes of finding a lower bound for $f(m)$ as discussed in \fullref{CSMachine}, one can work with as small or as large a set of relations as one desires.
Notice that if the set of relations is small, the resulting lower bound runs the risk of being too small as well and not be useful in any shape for future applications. On the opposite side, if the set of relations being considered is too large (or badly curated) the resulting optimization question will become too complicated to solve. The approach prefered in this paper is to work with and optimize a midsize collection of displacement functions but to choose that collection in such a way that the resulting lower bound is actually identical to that of a much larger collection. \fullref{tab:2} lists, classifies and counts all relations in the large collection $\mathcal{G}$.
\begin{table}[H]
	\centering
	\raa{1.5}
	\begin{tabular}{cccccc}
		\toprule
		& $\gamma_r$                                       & $\psi_r$                                 & $\Psi_r$                                                                               & $\{\cdot\}$                    & count          \\
		\midrule
		type 1a         & $\xi_{i_0}^{-t_0}$                               & $\xi_{i_0}^{2t_0}$                       & $\Psi-\{\xi_{i_0}^{2t_0},\ \xi_{i_0}^{t_0}\xij^{2s},\ \xi_{i_0}^{t_0}\xij^{s}\xik^p\}$ & $\Gamma_1-\{\xi_{i_0}^{t_0}\}$ & $2n$           \\
		type 1b         & $\xi_{i_0}^{t_0}\xi_{j_0}^{s_0}\xi_{i_0}^{-t_0}$ & $\xi_{i_0}^{2t_0}$                       & $\Psi-\{\xigent{i_0}{t_0}{j_0}{s_0}{i_0}{t_0}\}$                                       & $\Gamma_1$                     & $4n(n-1)$      \\
		\midrule
		type 2a         & $\xi_{i_0}^{-t_0}$                               & $\xigen{i_0}{t_0}{j_0}{2s_0}$            & $\Psi-\{\xi_{j_0}^{2s_0}\}$                                                            & $\Gamma_1$                     & $4n(n-1)$      \\
		type 2b         & $\xigent{i_0}{t_0}{j_0}{-s_0}{i_0}{-t_0}$        & $\xigen{i_0}{t_0}{j_0}{2s_0}$            & $\Psi-\{\xigen{i_0}{t_0}{j_0}{2s_0},\ \xigent{i_0}{t_0}{j_0}{s_0}{k}{p}\}$             & $\Gamma_1$                     & $4n(n-1)$      \\
		\midrule
		type 3a         & $\xi_{i_0}^{-t_0}$                               & $\xigent{i_0}{t_0}{j_0}{s_0}{i_0}{t_0}$  & $\Psi-\{\xigen{j_0}{s_0}{i_0}{2t_0},\ \xigent{j_0}{s_0}{i_0}{t_0}{k}{p}\}$             & $\Gamma_1$                     & $4n(n-1)$      \\
		type 3b         & $\xigent{i_0}{t_0}{j_0}{-s_0}{i_0}{-t_0}$        & $\xigent{i_0}{t_0}{j_0}{s_0}{i_0}{t_0}$  & $\Psi-\{\xi_{i_0}^{2t_0}\}$                                                            & $\Gamma_1$                     & $4n(n-1)$      \\
		\midrule
		type 4a         & $\xi_{i_0}^{-t_0}$                               & $\xigent{i_0}{t_0}{j_0}{s_0}{i_0}{-t_0}$ & $\Psi-\{\xigen{j_0}{s_0}{i_0}{-2t_0},\ \xigent{j_0}{s_0}{i_0}{-t_0}{k}{p}\}$           & $\Gamma_1$                     & $4n(n-1)$      \\
		type 4b         & $\xigent{i_0}{t_0}{j_0}{-s_0}{i_0}{-t_0}$        & $\xigent{i_0}{t_0}{j_0}{s_0}{i_0}{-t_0}$ & $\{\xi_{i_0}^{2t_0},\ \xigen{i_0}{t_0}{j}{2s},\ \xigent{i_0}{t_0}{j}{s}{k}{p}\}$       & $\{\xi_{i_0}^{t_0}\}$          & $4n(n-1)$      \\
		\midrule
		type 5a         & $\xi_{i_0}^{-t_0}$                               & $\xigent{i_0}{t_0}{j_0}{s_0}{k_0}{p_0}$  & $\Psi-\{\xigen{j_0}{s_0}{k_0}{2p_0},\ \xigent{j_0}{s_0}{k_0}{p_0}{k}{p}\}$             & $\Gamma_1$                     & $8n(n-1)(n-2)$ \\
		type 5b         & $\xigent{i_0}{t_0}{j_0}{-s_0}{i_0}{-t_0}$        & $\xigent{i_0}{t_0}{j_0}{s_0}{k_0}{p_0}$  & $\Psi-\{\xigen{i_0}{t_0}{k_0}{2p_0},\ \xigent{i_0}{t_0}{k_0}{p_0}{k}{p}\}$             & $\Gamma_1$                     & $8n(n-1)(n-2)$ \\
		\midrule
		$|\mathcal{G}|$ &                                                  &                                          & \multicolumn{3}{r}{$2(2n+4n(n-1)+8n(n-1)^2)+2n(2n-3)$}                                                                                   \\
		\bottomrule
	\end{tabular}
	\caption{Relations in the collection $\mathcal{G}$.}
	\label{tab:2}
\end{table}

In fact, $\mathcal{G}$ is the largest possible collection that can be worked with. This table has been prepared in a $\psi_r$ centric way, in the sense that to generate a relation, first $\psi_r$ is fixed and then a suitable $\gamma_r$ is chosen.
The calculation of $\Psi_r$ and $\{\cdot\}$ at every type is very similar to the example done in \fullref{DefNot}. With this collection at hand, we can define the function $G:\simp \longrightarrow \mathbb{R}_{++} $ as follows
\begin{align*}
	G(x) & =\max_{r\in\mathcal{G}}f_r(x).
\end{align*}
It is the infimum of this function that we will ultimately obtain in \fullref{sec:Optimization}.
The midsize collection $\mathcal{F}$ is the subset of $\mathcal{G}$ consisting of Type 1a, 2b, 3a, 4b and 5a relations.
The relevant function $F:\simp  \longrightarrow \mathbb{R}_{++} $ of this collection is defined as follows
\begin{align*}
	F(x)     & =  \max_{r\in\mathcal{F}}f_r(x).
\end{align*}
In \fullref{sec:Optimization}, we will mostly be optimizing $F$.

The choice of $\mathcal{F}$ can be explained as follows: For example, for any type 1b relation there exists a type 1a relation such that the type 1a relation's displacement function is larger at every point of the simplex.
This implies that the displacement functions for all the type 1b relations are unnecessary inside the max function.
Alas, most of the other "cancellations" aren't this straightforward and ultimately unnecessary as once the unique optimal solution for the infimum of $F$ is calculated, it is straight forward to show that it must also be the unique optimal point for the infimum of $G$. As a result, the collection $\mathcal{F}$ has been chosen largely due to the patterns observed in the $n=2$ in \cite{Yu4} and  the computer models observed by the authors. \fullref{tab:3} is listing the elements of $\mathcal{F}$ for ease of access.

\begin{table}[h]
	\centering
	\raa{1.7}
	\begin{tabular}{cclccccc}
		\toprule
		& $\psi_r$                                 & \multicolumn{5}{c}{$\Psi_r$} & count                                      \\
		\midrule
		1a                                        & $\xi_{i_0}^{2t_0}$                       &
		$\Psi-\{\xi_{i_0}^{2t_0},$                &
		$\xi_{i_0}^{t_0}\xij^{2s},$               &
		$\xigent{i_0}{t_0}{j}{s}{i_0}{t_0},$      &
		$\xigent{i_0}{t_0}{j}{s}{i_0}{-t_0},$     &
		$\xigent{i_0}{t_0}{j}{s}{k}{p}\ \}$       &
		$2n$                                                                                                                                                             \\
		
		&                                          &
		\multicolumn{1}{r}{$1$}                   &
		$2(n-1)$                                  &
		$2(n-1)$                                  &
		$2(n-1)$                                  &
		$4(n-1)(n-2)$                             &
		\\
		\midrule
		
		2b                                        & $\xigen{i_0}{t_0}{j_0}{2s_0}$            & $\Psi-\{$                    &
		$\xigen{i_0}{t_0}{j_0}{2s_0},$            &
		$\xigent{i_0}{t_0}{j_0}{s_0}{i_0}{t_0},$  &
		$\xigent{i_0}{t_0}{j_0}{s_0}{i_0}{-t_0},$ &
		$\xigent{i_0}{t_0}{j_0}{s_0}{k}{p}\ \}$   &
		$4n(n-1)$                                                                                                                                                        \\
		
		&                                          &
		\multicolumn{1}{r}{ }                     &
		$1$                                       &
		$1$                                       &
		$1$                                       &
		$2(n-2)$                                  &
		\\
		\midrule
		
		3a                                        & $\xigent{i_0}{t_0}{j_0}{s_0}{i_0}{t_0}$  & $\Psi-\{$                    &
		$\xigen{j_0}{s_0}{i_0}{2t_0},$            &
		$\xigent{j_0}{s_0}{i_0}{t_0}{j_0}{s_0},$  &
		$\xigent{j_0}{s_0}{i_0}{t_0}{j_0}{-s_0},$ &
		$\xigent{j_0}{s_0}{i_0}{t_0}{k}{p}\ \}$   &
		$4n(n-1)$                                                                                                                                                        \\
		&                                          &
		\multicolumn{1}{r}{ }                     &
		$1$                                       &
		$1$                                       &
		$1$                                       &
		$2(n-2)$                                  &
		\\
		\midrule
		4b                                        & $\xigent{i_0}{t_0}{j_0}{s_0}{i_0}{-t_0}$ &
		\multicolumn{1}{r}{$\{\xi_{i_0}^{2t_0},$} &
		$\xigen{i_0}{t_0}{j}{2s},$                &
		$\xigent{i_0}{t_0}{j}{s}{i_0}{t_0},$      &
		$\xigent{i_0}{t_0}{j}{s}{i_0}{-t_0},$     &
		$\xigent{i_0}{t_0}{j}{s}{k}{p}\ \}$       &
		$4n(n-1)$                                                                                                                                                        \\
		&                                          &
		\multicolumn{1}{r}{$1$}                   &
		$2(n-1)$                                  &
		$2(n-1)$                                  &
		$2(n-1)$                                  &
		$4(n-1)(n-2)$                             &
		\\
		\midrule
		5a                                        & $\xigent{i_0}{t_0}{j_0}{s_0}{k_0}{p_0}$  &
		$\Psi-\{$                                 &
		$\xigen{j_0}{s_0}{k_0}{2p_0}$             &
		$\xigent{j_0}{s_0}{i_0}{t_0}{j_0}{s_0},$  &
		$\xigent{j_0}{s_0}{i_0}{t_0}{j_0}{-s_0},$ &
		$\xigent{j_0}{s_0}{i_0}{t_0}{k}{p}\ \}$   &
		$\displaystyle\frac{8n!}{(n-3)!}$                                                                                                                                             \\
		&                                          &
		\multicolumn{1}{r}{ }                     &
		$1$                                       &
		$1$                                       &
		$1$                                       &
		$2(n-2)$                                  &
		\\
		\midrule
		$|\mathcal{F}|$                           &                                          &
		&                                          &                              & \multicolumn{3}{r}{$2n+4n(n-1)+8n(n-1)^2$} \\
		\bottomrule
	\end{tabular}
	\caption{Relations in the collection $\mathcal{F}$.}
	\label{tab:3}
\end{table}

In \fullref{tab:3}, we omitted the set $\{\cdot\}$ and $\gamma_r$ columns (as they are irrelevant to the calculation of displacements) and expanded the description of $\Psi_r$ as to distinguish (and count) type 3, 4 and 5 elements of $\Psi_r$ better (see \fullref{tab:1}). Also notice that $|\mathcal{F}|=|\Psi|$. 

In the rest of this section, we shall also calculate the gradients of the displacement functions defined in \fullref{DefNot}.
They will be needed in \fullref{sec:Optimization}. 
Recall that for a relation $r$, the definition of the displacement function $f_r$ of $r$ is 
\begin{equation*}
f_r(x)=\frac{(1-x_r)(1-X_r)}{x_rX_r},
\end{equation*}
 where $x_r=x(\psi_r)$ and $X_r=\sum_{\psi\in\Psi_r}x(\psi)$.
Thus, the partial derivative of $f_r$ with respect to a coordinate $\psi\in\Psi$ depends on whether the coordinate is in $\Psi_r$ and if it is equal to $\psi_r$ or not.
To further complicate the issue, we have $\psi_r\in\Psi_r$ for some displacements and $\psi_r\notin\Psi_r$ for some displacements. Therefore, we present the gradient of $f_r$ for a given relation $r$ in two cases. For relations of type 1a and 2b, we have $\psi_r\notin\Psi_r $ which implies that
\begin{eqnarray*}
 \frac{d f_r}{d \psi} &=&\begin{cases}
			\displaystyle\frac{1-x_r}{x_r}\cdot\frac{-1}{X_r^2} & \text{ if } \psi\in\Psi_r                  \\
			\displaystyle\frac{-1}{x_r^2}\cdot\frac{1-X_r}{X_r} & \text{ if } \psi=\psi_r                    \\
			0                                      & \text{ if } \psi\notin\Psi_r\cup\{\psi_r\}.
		\end{cases}
\end{eqnarray*}
Using the partial derivatives above, we see that  
\begin{equation}\label{grad1}
\nabla f_r(x)=-\frac{1-x_r}{x_rX_r^2}\cdot\sum_{\psi\in\Psi_r}e_\psi - \frac{1-X_r}{x_r^2X_r}\cdot e_{\psi_r}.
\end{equation}
For relations of type 3a, 4b and 5a, we have $\psi_r\in\Psi_r $ so that 
\begin{eqnarray*}
 \frac{d f_r}{d \psi} &=&\begin{cases}
			\displaystyle\frac{1-x_r}{x_r}\cdot\frac{-1}{X_r^2}                                        & \text{ if } \psi\in\Psi_r-\{\psi_r\}       \\
			\displaystyle\frac{-1}{x_r^2}\cdot\frac{1-X_r}{X_r}+\frac{1-x_r}{x_r}\cdot\frac{-1}{X_r^2} & \text{ if } \psi=\psi_r                    \\
			0                                                                             & \text{ if } \psi\notin\Psi_r\cup\{\psi_r\}.
		\end{cases}
\end{eqnarray*}
As a result, we calculate the gradients as 
\begin{equation}\label{grad2}
 \nabla f_r(x)=-\frac{1-x_r}{x_rX_r^2}\cdot\sum_{\psi\in\Psi_r}e_\psi - \left(\frac{1-X_r}{x_r^2X_r}+\frac{1-x_r}{x_rX_r^2}\right)\cdot e_{\psi_r}.
\end{equation}
In both (\ref{grad1}) and (\ref{grad2}),  the expression $(e_\psi)_\Psi$ is the canonical basis for $\mathbb{R}^\Psi$.

%%%%%%%%%%%%%%%%%%%%%%%%%%%%%%%%%%%%% SECTION 5

\section{Optimizing the Infimum of the Max Function} \label{sec:Optimization}
The purpose of this section is to compute $\inf_{x\in\simp}(G(x))$. But for this purpose, we will calculate $\inf_{x\in\simp}(F(x))$.
This calculation can be summarized in three phases.
The first phase is to show the attainment of optimality (for both $G$ and $F$). The second phase is to prove the uniqueness of optimality. The final phase is to calculate the point of optimality using the Karush-Kuhn-Tucker theorem.
We would like to point out that the problem with the infimum is not that of existence, as the image is a bounded below subset of $\mathbb{R}$ that part is obvious, but attainment in $\simp$.
In fact, from here on we shall denote the infimums with $\alpha^*$ and $\beta^*$; in other words,
\begin{equation*}
	\alpha^*=\inf_{x\in\simp}(G(x))\ \textnormal{ and }\ \beta^*=\inf_{x\in\simp}(F(x)).
\end{equation*}

\begin{lemma}
	There exists $x^*\in\simp$ such that $G(x^*)=\alpha^*$.
\end{lemma}

\begin{proof}
	As $\alpha^*$ is the infimum of the set $G(\simp)$, there must be a sequence $(x^n)_n$ in $\simp$ such that $G(x^n)\rightarrow\alpha^*$.
	Moreover, $(x^n)_n$ must have a convergent subsequence since the closure of $\simp$ is a bounded subset of $\mathbb{R}^\Psi$. Without loss of generality, assume $x^n\rightarrow \bar{x}$ for some $\bar{x}\in cl(\simp)$.
	Notice that $cl(\simp)=\{x\in\mathbb{R}^\Psi_+\ |\ X=1\}$.
	The fact that the max function of continuous functions is also continuous is a simple exercise (\cite{convex} Exercise 2.3.1).
	Our aim is to prove $\bar{x}\notin cl(\simp)-\simp$. We shall do so by assuming the opposite and obtaining a contradiction.
	
	Let $K=\set{\psi\in\Psi}{\bar{x}_\psi=0}$.
	The assumption that $\bar{x}\in cl(\simp)-\simp$ implies $K\neq\emptyset$.
	Now, if there exists a relation $r$ such that $\psi_r\in K$ and $\overline{X}_r\in [0,1)$, then $x^n_r\rightarrow \bar{x}_r=0$ and $({1-X^n_r})/{X^n_r}$ either converges or diverges to $\infty$,
	which shows that $f_r(x^n)=(1-x^n_r)/x^n_r\cdot (1-X^n_r)/X^n_r\rightarrow \infty$.
	However, $f_r(x)\leq G(x)$ for any $x$. Thus, we also have $\lim f_r(x^n)\leq \lim G(x^n)=\alpha^*$, contradicting the existence of $\alpha^*$.
	
	The above paragraph shows that for all $r\in\mathcal{G}$ with $\psi_r\in K$, we have $\overline{X}_r=1$.
	However, we know that $\overline{X}_r+\sum_{\psi\notin\Psi_r}\bar{x}(\psi)=\overline{X}=1$ proving $\bar{x}(\psi)=0$ for all $\psi\in\Psi-\Psi_r$ in any such relation.
	In summary, what we showed is that 
	\begin{align}
		\label{Keq:1}
		\psi_r\in K \textnormal{ implies } \Psi-\Psi_r \subseteq K
	\end{align}
	for a relation $r\in\mathcal{G}$. The contradiction that we seek will soon come in the form of $K=\Psi$ as that would imply $\overline{X}=0$, contradicting the fact that $\bar{x}\in cl(\simp)$.
We already know $K\neq \emptyset$ by our assumption. We shall start with a random element of $K$. We will  cycle over all possible types in \fullref{tab:1} and use (\ref{Keq:1}) together with \fullref{tab:2} to show that $K=\Psi$.
	
	If $\xi_{i_0}^{2t_0}\in K$, the relevant type 1a relation $r_0$ will satisfy $\psi_{r_0}\in K$. By (\ref{Keq:1}), we get $\Psi-\Psi_{r_0}\subseteq K$, implying $\{\xigent{i_0}{t_0}{j}{s}{i_0}{-t_0}\}\subseteq K$.
	Repeating the same process for the type 4b relation $r_j$ with $\psi_{r_j}=\xigent{i_0}{t_0}{j}{s}{i_0}{-t_0}$, we get $\{\xi_i^{2t}\}-\{\xi_{i_0}^{2t_0}\}\subseteq K$.
	Since we already know that $\xi_{i_0}^{2t_0}\in K$, we actually have $\{\xi_i^{2t}\}\subseteq K$. The arguments so far in this paragraph show that for any type 1a relation $r$ we obtain $\psi_r\in K$.
	Merging all this information together, we get $\Psi=\bigcup_{i=1}^{n}(\Psi-\Psi_{r_i})\subseteq K$, proving what we wanted. We summarize this result below:
	\begin{equation}
		\label{Keq:2}
		\text{if there exists a type $1$ element in  $K$, then}\ \Psi=K.
	\end{equation}
	
	If $\xigen{i_0}{t_0}{j_0}{2s_0}\in K$, the relevant type 2b relation $r_0$ will satisfy $\psi_{r_0}\in K$ and so by (\ref{Keq:1}) we get $\Psi-\Psi_{r_0}\subseteq K$, implying $\{\xigent{i_0}{t_0}{j_0}{s_0}{i_0}{-t_0}\}\subseteq K$.
	Repeating the same process for the type 4b relation $r_1$ with $\psi_{r_1}=\xigent{i_0}{t_0}{j_0}{s_0}{i_0}{-t_0}$, we get $\{\xi_i^{2t}\}-\{\xi_{i_0}^{2t_0}\}\subseteq K$. This shows that $K$ has a type 1 element and $K=\Psi$ by (\ref{Keq:2}).
	We summarize this result below:
	\begin{equation}
		\label{Keq:3}
		\text{if there exists a type $2$ element in  $K$, then}\ \Psi=K.
	\end{equation}
	
If $K$ has an element of the other types, we have the same conclusion $\Psi=K$. More specifically, we see that
	 if $K$ has a type $3$ element and the relevant type $3a$ relation together with (\ref{Keq:3}), then $K=\Psi$. The assumption that
	$K$ has a type $4$ element and the relevant type 4b relation together with (\ref{Keq:2}) implies $K=\Psi$. If 
		$K$ has a type 5 element and the relevant type 5a relation together with (\ref{Keq:3}), then $K=\Psi$.
		This gives that as long as $K$ has an element no matter what type, $K=\Psi$.
	\end{proof}

Notice that in the proof above, we only used the relations in $\mathcal{F}$. Therefore, the following lemma has an identical proof.
\begin{lemma}
	\label{Existence}
	There exists $x^*\in\simp$ such that $F(x^*)=\beta^*$.
\end{lemma}

The next lemma gives us a rough location for $\beta^*$ and $\alpha^*$ which proves very useful for the proofs in the second phase.

\begin{lemma}\label{AlphaIsBound}
	The infimums of $F(x)$ and $G(x)$ on $\simp$, $\beta^*$ and $\alpha^*$ are greater than $1$ and less than or equal to $\alpha$. Moreover, $\beta^*\leq \alpha^*$.
\end{lemma}
\begin{proof}
	First of all, we know that $1<\alpha$ by \fullref{Proots}. Therefore, at the very least the lemma makes sense.
	Also, since $\F\subset \G$, we have $F(x)\leq G(x)$ for any $x\in \simp$.
	This proves $\beta^*\leq \alpha^*$.
	
	To show that $1<\beta^*$, let $r$ be any type 1a relation.
	As $\psi_r\notin\Psi_r$ we have $1=X>x_r+X_r$ for all $x\in\simp$.
	Dividing both sides of this inequality with $x_rX_r$, we obtain:
	\[\frac{1}{x_rX_r}>\frac{1}{x_r}+\frac{1}{X_r}\ \textnormal{ or }  \frac{1}{x_rX_r}-\frac{1}{x_r}-\frac{1}{X_r}>0.\]
	Moreover, we have
	\begin{equation*}f_r(x)=\frac{1-x_r}{x_r}\cdot \frac{1-X_r}{X_r}=\frac{1}{x_rX_r}-\frac{1}{x_r}-\frac{1}{X_r}+1,\end{equation*}
	which must be greater than 1 by the first inequality.
	Since $F(x)\geq f_r(x)$ for any relation (and specifically any type 1a relation) in $\mathcal{F}$, we see that $\beta^*=F(x^*)\geq f_r(x^*)>1$ for $x^*$ in \fullref{Existence}.
	To show that $\alpha^*\leq \alpha$, we shall need a special element $y\in\simp$. Define
	\[
	y(\psi) =
	\begin{cases}
		\displaystyle\frac{1}{(2n-1)\alpha+1}                                                  & \text{if $\psi$ is type 1},     \\
		\displaystyle\frac{(2n-1)(\alpha-1)}{(4n^2-4n-1)(2n-1)\alpha^2+(4n^2-2)\alpha -(2n-1)} & \text{if $\psi$ is type 2,3,5,} \\
		\displaystyle\frac{(2n-1)}{(2n-1)+\alpha}                                              & \text{if $\psi$ is type 4}.
	\end{cases}
	\]
	All the coordinates of $y$ are obviously greater than $0$. It is straight forward to show that $Y=1$ by using \fullref{tab:1} and the counts of element types therein.
	The resulting expression can easily be simplified using $\Pl(\alpha)=0$.
	Similarly, one can also show that $f_r(y)=\alpha$ for all $r\in\F$ and $f_r(y)\leq\alpha$ for all $r\in\G$ by using \fullref{tab:3} and the counts of element types within every $\Psi_r$. Therefore $\alpha^*=\inf(G(x))\leq G(y)=\alpha$. In summary, we have $1< \beta^*\leq \alpha^* \leq \alpha$, completing the proof.
\end{proof}

The previous lemma marks the end of phase one.
The vector $y$ given in the proof is in fact the unique optimization point for our problem. Proving this statement takes most of the technical work in this paper.
What we shall do next is to prove the uniqueness of the optimal point using convexity arguments via a bank of lemmas.

\begin{lemma}\label{Some4bAttains}
	For all $x^*\in\simp$ such that $F(x^*)=\beta^*$, there exists a type 4b relation $r_0$ so that $f_{r_0}(x^*)=\beta^*$.
\end{lemma}

\begin{proof}
	If we consider $\simp$ as a submanifold of $\mathbb{R}^\Psi$, the tangent space $T_x\simp$ at any point $x\in\simp$ is the set of vectors whose coordinates sum to $0$.
	Moreover, every displacement function is smooth on an open set containing $\simp$ which implies that the directional derivative for $\vec{u}\in T_x\simp$ will be $\nabla f(x)\cdot \vec{u}$.
	
	Our plan of attack is to find a common direction of decrease for type 1a, 2b, 3a and 5a relations, which will give us a very nice contradiction if the infimum was not attained by any type 4b relation at an optimal point. As the common direction of decrease, consider the vector $\vec{u}$ described below.
	\begin{equation*}
		u(\psi)=
		\begin{cases}
			\ \ \ \ \ \ 1              & \ \ 	\text{if $\psi$ is not type 4,} \\
			\displaystyle-\frac{1+4(n-1)^2}{2(n-1)} & \ \ 	\text{if $\psi$ is type 4}.
		\end{cases}
	\end{equation*}
	It is easy to check that the sum of coordinates of $\vec{u}$ is $0$ with the help of \fullref{tab:1}, in turn, this shows that $\vec{u}\in T_x\simp$.
	Moreover, we present the calculations below using (\ref{grad1}),  (\ref{grad2}) and \fullref{tab:3} proving that directional derivatives of type 1a, 2b, 3a and 5a relation displacements are all negative in the direction $\vec{u}$.
		For $r$ of type 1a, we see that
		\begin{align*}
			\nabla f_r(x)\cdot \vec{u} & =-\frac{1-X_r}{x_r^2X_r}\cdot u_{\psi_r}-\frac{1-x_r}{x_rX_r^2}\sum_{\psi\in\Psi_r}u_\psi                                                                       \\
			& =-\frac{1-X_r}{x_r^2X_r}\cdot 1-\frac{1-x_r}{x_rX_r^2}\left(\sum_{\psi\in\Psi_r-\text{type 4}}u_\psi+\sum_{\psi\in\Psi_r\cap\text{type 4}}u_\psi\right)         \\
			& =-\frac{1-X_r}{x_r^2X_r}-\frac{1-x_r}{x_rX_r^2}\left(\sum_{\psi\in\Psi_r-\text{type 4}}1+\sum_{\psi\in\Psi_r\cap\text{type 4}}-\frac{1+4(n-1)^2}{2(n-1)}\right) \\
			& =-\frac{1-X_r}{x_r^2X_r}-\frac{1-x_r}{x_rX_r^2}\left( (2n-1)(1+4(n-1)^2)-\frac{1+4(n-1)^2}{2(n-1)}\cdot2(n-1)(2n-1)\right)                                      \\
			& =-\frac{1-X_r}{x_r^2X_r}                                                                                                                                        \\
			& <0.
		\end{align*}
		For the last inequality, recall that $x_r,\ X_r \in (0,1)$.
		 For type 2b, 3a and 5a, we find that
		\begin{align*}
			\nabla f_r(x)\cdot \vec{u} & =-\frac{1-X_r}{x_r^2X_r}-\frac{1-x_r}{x_rX_r^2}\cdot \left(1-\frac{1}{2(n-1)}\right)<0.
		\end{align*}

	Assume that there exists an optimal point $x^*$ such that $f_r(x^*)<\beta^*$ for all $r$ of type 4b.
	Let $\varepsilon>0$ be small enough so that $x^*+\varepsilon \vec{u}\in \simp$ and $f_r(x^*+\varepsilon \vec{u})<\beta^*$ for all $r$ of type 4b.
	Such a choice is possible by the continuity of displacement functions.
	Under these conditions, we have  $F(x^*+\varepsilon \vec{u})<\beta^*$ as $f_r(x^*+\varepsilon \vec{u})<f_r(x^*)\leq F(x^*)=\beta$ for all $r$ of type 1a, 2b, 3a, 5a and $f_r(x^*+\varepsilon \vec{u})<\beta$ for all $r$ of type 4b, contradicting the minimality of $\beta^*$.
\end{proof}

Unfortunately, the displacement functions are not convex on the entire simplex. Therefore, neither are the max functions.
Any argument of uniqueness of the optimal solution must first be defended on a local level, displacement by displacement.
Fortunately, each displacement function is strictly convex on a relatively large set.

\begin{lemma}
	For $r\in \F$ define $C_r=\set{x\in \simp}{x_r+X_r-x_rX_r<3/4}$.
	The displacement function $f_r$ is strictly convex on $C_r$.
\end{lemma}

\begin{proof}
	Consider the function $g:(0,1)^2\rightarrow \mathbb{R}$ defined by $g(z_1,z_2)=({1-z_1})/{z_1}\cdot ({1-z_2})/{z_2}$. As $g$ is twice differentiable, $g$ is strictly convex, where the Hessian of $g$ is positive definite by \cite[Theorem 3.1.11 and Exercise 3.1.12(c)]{convex}:
	\[\text{Hess}(g)=
	\begin{pmatrix}
		2(1-z_2)/z_1^3z_2 & 1/z_1^2z_2^2          \\
		1/{z_1^2z_2^2}    & {2(1-z_1)}/{z_1z_2^3}
	\end{pmatrix}.
	\]
	Moreover, a $2\times 2$ matrix is positive definite if and only if $\det(H)>0$ and $\text{trace}(H)>0$, where
	\begin{align*}
		\det(\text{Hess}(g)) =\frac{4(1-z_1)(1-z_2)}{z_1^4z_2^4}-\frac{1}{z_1^4z_2^4}\quad &\ \textnormal{and}\quad \text{trace}(\text{Hess}(g))  =\frac{2(1-z_1)}{z_1z_2^3}+\frac{2(1-z_2)}{z_1^3z_2}.
	\end{align*}
	Since the trace of the Hessian is always positive in our domain, $g$ is strictly convex, where the determinant is positive.
	Hence, $g$ is strictly convex on $C_g=\set{(z_1,z_2)\in(0,1)^2}{z_1+z_2-z_1z_2<3/4}$. Now, notice that $f_r=g(x_r,X_r)$. Since the map $x\ra (x_r,X_r)\in(0,1)^2$ satisfies the conditions listed in \cite[Exercise 1.1.9]{convex},  we get $f_r$ is strictly convex on $C_r=\set{x\in \simp}{x_r+X_r-x_rX_r<3/4}$.
\end{proof}

We already know that some type 4b displacement attains $\beta^*$ at a point of optimality by \fullref{Some4bAttains}.
It is also true that any optimal point of a strictly convex function is unique \cite[Exercise 2.1.8(a)]{convex}.
Therefore, proving any optimal point must be in the domain of convexity of all type 4b relation displacements will take us most of the way to the uniqueness.

\begin{theorem}
	If $r$ is a type 4b relation, then $x^*\in C_r$ for all $x^*\in\simp$ such that $F(x^*)=\beta^*$.
\end{theorem}

\begin{proof}
	Assume there exists a point of optimality $x^*$  such that $x^*\notin C_{r_0}$ for a type 4b group theoretical relation $r_0:\ \ \psi_0=\xigent{i_0}{t_0}{j_0}{s_0}{i_0}{-t_0},$ $\Psi_0=\{\xi_{i_0}^{2t_0},$ $\xigen{i_0}{t_0}{j}{2s},\ \xigent{i_0}{t_0}{j}{s}{k}{p}\}$, %, ie,
where	
	\begin{equation}\label{AboutC}
		x^*_{0}+X^*_{0}-x^*_{0}X^*_{0}\geq 3/4.
	\end{equation}
	Our aim is to obtain a contradiction under these conditions.
	Let $T$ be the collection of type 4b relations such that $\psi_r=\xigent{i}{t}{j_0}{s_0}{i}{-t}$ and $\psi_r=\xigent{j_0}{s}{i_0}{t_0}{j_0}{-s}$.
	Notice that $(\Psi_r)_{r\in T}$ gives a disjoint partition of $\Psi$.
	This means that
	$\displaystyle \Psi=\bigsqcup_{r\in T}\Psi_r,$
	which implies that for any $x\in\simp$, we get
	\begin{equation}\label{About1}
		1=\sum_{r\in T}X_r.
		\end{equation} 
		Based on the possible location of $x^*$, we will separate the proof into the following three cases 
\begin{equation*} (A)\ \frac{1}{2} \geq X^*_0,\quad
		(B)\ \frac{1}{2n+2} \geq X^*_{r_1} \text{ for some }r_1\in T,\quad
		 (C)\ \frac{1}{2}<X^*_0 \text{ and } \frac{1}{2n+2}<X^*_{r} \text{ for all }r\in T.
\end{equation*}
Each one of these cases will end up with a contradiction, proving the theorem. Assume $(A)$ is the case. Then, (\ref{AboutC}) gives us the following chain of inequalities:
	\begin{equation*}
		 X^*_0(1-x^*_0)+x^*_0\geq 3/4\quad \textnormal{or}\quad \frac{1}{2}(1-x^*_0)+x^*_0 \geq 3/4\quad \textnormal{or}\quad x^*_0 \geq 1/2.
	\end{equation*}
	However, as $\psi_0\in\Psi_0$, we get the contradiction $1/2\geq X^*_0 >x^*_0 \geq 1/2$. 
	
	Assume $(B)$ is the case. Let $\psi_{r_1}=\xigent{i_1}{t_1}{j_1}{s_1}{i_1}{-t_1}$. Consider any type 4b relation $r$ with $\psi_r=\xigent{i_1}{t_1}{j}{s}{i_1}{-t_1}$. By \fullref{AlphaIsBound}, we have
	\begin{equation*}
		\frac{1-x^*_{r}}{x^*_{r}}\cdot \frac{1-X^*_{r}}{X^*_{r}} = f_{r}(x^*) \leq \beta^* \leq \alpha.
	\end{equation*}
	This, after a number of re-arrangements, turns into the following inequality
	\begin{equation*}
		\frac{1-X^*_{r}}{X^*_{r}(\alpha-1)+1} \leq x^*_{r}.
	\end{equation*}
	Notice that $\Psi_{r_1}=\Psi_r$ and so $1/(2n+1) \geq X^*_{r_1}=X^*_r$. In the above inequality, the left side is a decreasing function of $X^*_{r}$. Thus, by the case assumption, we preserve the inequality by replacing $X^*_{r}$ with $1/(2n+2)$. We obtain
	\begin{equation*}
		\frac{2n+1}{\alpha+2n+1} = \frac{1-\displaystyle\frac{1}{2n+2}}{\displaystyle\frac{1}{2n+2}(\alpha-1)+1} \leq x^*_{r}=x^*(\xigent{i_1}{t_1}{j}{s}{i_1}{-t_1}).
	\end{equation*}
	The following inequalities 
	\begin{equation*}
		\frac{1}{2n+2}\geq X^*_{r_1}>\sum_{j\neq i_1,\ s=\pm1}x^*(\xigent{i_1}{t_1}{j}{s}{i_1}{-t_1}) \geq 2(n-1)\cdot\frac{2n+1}{\alpha+2n+1}
	\end{equation*}
	are obtained by using the previous inequality and the fact that $\xigent{i_1}{t_1}{j}{s}{i_1}{-t_1}\in \Psi_{r_1}$ for all $j\neq i_1$, $s=\pm1$. Again, after a number of re-arrangements, we get
			$\alpha>(2n-1)^3+8n^2-12n-4.$
	However, we have $8n^2-12n-3>0$ for $n\geq 2$ implying $\alpha>(2n-1)^3$, contradicting the proof of \fullref{Proots}.
	
	Assume $(C)$ is the case. By using (\ref{About1}) and the fact that $n\geq 2$, we derive that
	\begin{align*}
		1=X^*_0+\sum_{r\in T-\{r_0\}}X^*_r > \frac{1}{2} + (2n-1)\cdot \frac{1}{2n+2}=1+\frac{1}{2}-\frac{3}{2n+2} \geq 1,
	\end{align*}
	which is the final contradiction that we needed.
\end{proof}

It is a simple exercise to show that the max function of strictly convex functions is strictly convex.
Hence, we have the fact that $F$ is strictly convex on $\displaystyle\cap_{r\in \text{Type 4b}}C_r$ which is a non empty set as it contains all the optimal points at the very least.
As optimal points of strictly convex functions are unique, we have proven the following corollary.

\begin{corollary}\label{Uniquness}
	The optimal point of $\text{inf}_{x\in\simp}F(x)$ is unique.
\end{corollary}
Henceforth, we shall use $x^*$ to denote the unique optimal point. In other words, we let
\begin{equation*}
	x^*:\ \text{ the unique point in } \simp \text{ with } F(x^*)=\beta^*,
\end{equation*}
which marks the end of phase two.

We start phase three by simplifying $x^*$. The uniqueness of $x^*$ will be heavily used in this simplification. Then, we apply the Karush-Kuhn-Tucker theorem. 

\begin{lemma}\label{SymSimplification}
	There exists $a^*,\ b^*,\ c^*\in (0,1)$ such that
	\begin{equation*}
		x^*(\psi)=\left.
		\begin{cases}
			a^* & \text{for } \psi \in \text{type 1}                     \\
			b^* & \text{for } \psi \in \text{type 2 or type 3 or type 5} \\
			c^* & \text{for } \psi \in \text{type 4}.
		\end{cases}\right.
	\end{equation*}
\end{lemma}

\begin{proof}
	Let $i_1, i_2\in \{1,2,\dots,n\}$ and $t_1,t_2\in\{-1,+1\}$ such that $i_1\neq i_2$.
	Consider the symmetry $\tau\in \text{Sym}(\Psi)$, which switches $\xi_{i_1}^{t_1}$ with $\xi_{i_2}^{t_2}$  and vice versa in every element of $\Psi$. For clarification, we give the details of the action of $\tau$ in \fullref{AboutTau}.
	
	\begin{table}[H]
		\centering
		\raa{2}
		\begin{tabular}{lllll}
			\toprule
			type 1 & $\xi_{i_1}^{\pm 2t_1} \leftrightarrow \xi_{i_2}^{\pm 2t_2}$                                               \\
			\midrule
			type 2 & $\xigen{i_1}{t_1}{i_2}{\pm2t_2} \leftrightarrow \xigen{i_2}{t_2}{i_1}{\pm2t_1}$
			& $\xigen{i_1}{-t_1}{i_2}{\pm2t_2} \leftrightarrow \xigen{i_2}{-t_2}{i_1}{\pm2t_1}$                         \\
			
			& $\xigen{i}{t}{i_1}{\pm2t_1} \leftrightarrow \xigen{i}{t}{i_2}{\pm2t_2}$
			& $\xigen{i_1}{\pm t_1}{j}{2s} \leftrightarrow \xigen{i_2}{\pm t_2}{j}{2s}$                                 \\
			\midrule
			type 3 & $\xigent{i_1}{t_1}{i_2}{\pm t_2}{i_1}{t_1} \leftrightarrow \xigent{i_2}{t_2}{i_1}{\pm t_1}{i_2}{t_2}$
			& $\xigent{i_1}{-t_1}{i_2}{\pm t_2}{i_1}{-t_1} \leftrightarrow \xigent{i_2}{-t_2}{i_1}{\pm t_1}{i_2}{-t_2}$ \\
			
			& $\xigent{i}{t}{i_1}{\pm t_1}{i}{t} \leftrightarrow \xigent{i}{t}{i_2}{\pm t_2}{i}{t}$
			& $\xigent{i_1}{\pm t_1}{j}{s}{i_1}{\pm t_1} \leftrightarrow \xigent{i_2}{\pm t_2}{j}{s}{i_2}{\pm t_2}$     \\
			\midrule
			type 4 & $\xigent{i_1}{t_1}{i_2}{\pm t_2}{i_1}{-t_1} \leftrightarrow \xigent{i_2}{t_2}{i_1}{\pm t_1}{i_2}{-t_2}$
			& $\xigent{i_1}{-t_1}{i_2}{\pm t_2}{i_1}{t_1} \leftrightarrow \xigent{i_2}{-t_2}{i_1}{\pm t_1}{i_2}{t_2}$   \\
			
			& $\xigent{i}{t}{i_1}{\pm t_1}{i}{-t} \leftrightarrow \xigent{i}{t}{i_2}{\pm t_2}{i}{-t}$
			& $\xigent{i_1}{\pm t_1}{j}{s}{i_1}{\mp t_1} \leftrightarrow \xigent{i_2}{\pm t_2}{j}{s}{i_2}{\mp t_2}$     \\
			\midrule
			type 5 & $\xigent{i_1}{t_1}{i_2}{\pm t_2}{l}{p} \leftrightarrow \xigent{i_2}{t_2}{i_1}{\pm t_1}{l}{p}$
			& $\xigent{i_1}{-t_1}{i_2}{\pm t_2}{l}{p} \leftrightarrow \xigent{i_2}{-t_2}{i_1}{\pm t_1}{l}{p}$           \\
			
			& $\xigent{i_1}{t_1}{j}{s}{i_2}{\pm t_2} \leftrightarrow \xigent{i_2}{t_2}{j}{s}{i_1}{\pm t_1}$
			& $\xigent{i_1}{-t_1}{j}{s}{i_2}{\pm t_2} \leftrightarrow \xigent{i_2}{-t_2}{j}{s}{i_1}{\pm t_1}$           \\
			
			& $\xigent{i}{t}{i_1}{t_1}{i_2}{\pm t_2} \leftrightarrow \xigent{i}{t}{i_2}{t_2}{i_1}{\pm t_1}$
			& $\xigent{i}{t}{i_1}{-t_1}{i_2}{\pm t_2} \leftrightarrow \xigent{i}{t}{i_2}{-t_2}{i_1}{\pm t_1}$           \\
			
			& $\xigent{i}{t}{j}{s}{i_1}{\pm t_1} \leftrightarrow \xigent{i}{t}{j}{s}{i_2}{\pm t_2}$
			& $\xigent{i}{t}{i_1}{\pm t_1}{l}{p} \leftrightarrow \xigent{i}{t}{i_2}{\pm t_2}{l}{p}$
			& $\xigent{i_1}{\pm t_1}{j}{s}{l}{p} \leftrightarrow \xigent{i_2}{\pm t_2}{j}{s}{l}{p}$                     \\
			\bottomrule
		\end{tabular}
		\caption{Actions of the symmetry $\tau\in\text{Sym}(\Psi)$}.
		\label{AboutTau}
	\end{table}
	One can prove that $\tau$ can also be considered as an element of $\text{Sym}(\F)$ in the sense that $\psi_{\tau(r)}=\tau(\psi_r)$ and $\Psi_{\tau(r)}=\tau(\Psi_r)$ by observing the effects of $\tau$ on the elements of $\F$ and using \fullref{tab:3}. As a result, we have the following
	\begin{equation*}
		F(\tau(x))=\text{max}_{r\in\F}f_r(\tau(x))=\text{max}_{r\in\F}f_{\tau(r)}(x)=\text{max}_{r\in\F}f_r(x)=F(x).
	\end{equation*}
	In particular, we have $F(\tau(x^*))=F(x^*)$ and hence $\tau(x^*)=x^*$ by \fullref{Uniquness}.
	By repeated application of $\tau$ (with different values for $(i_1,t_1)$ and$(i_2,t_2)$), we can show that any coordinate of $x^*$ is equal to any other coordinate of the same type. We give the following example of type 5 coordinates for clarification
	\begin{equation*}
		x^*(\xigent{i_1}{t_1}{j_1}{s_1}{l_1}{p_1}) \underset{\underset{(i_1,t_1)\rightarrow (i_2,t_2)}{\tau}}{=} x^*(\xigent{i_2}{t_2}{j_1}{s_1}{l_1}{p_1}) \underset{\underset{(j_1,s_1)\rightarrow (j_2,s_2)}{\tau}}{=} x^*(\xigent{i_2}{t_2}{j_2}{s_2}{l_1}{p_1}) \underset{\underset{(l_1,p_1)\rightarrow (l_2,p_2)}{\tau}}{=} x^*(\xigent{i_2}{t_2}{j_2}{s_2}{l_2}{p_2}).
	\end{equation*}
	This proves that there are $a^*,\ b^*_2,\ b^*_3,\ b^*_5,\ c^*\in (0,1)$ such that:
	\begin{equation*}
		x^*(\psi)=\left.
		\begin{cases}
			a^*   & \text{for } \psi \in \text{type 1} \\
			b^*_2 & \text{for } \psi \in \text{type 2} \\
			b^*_3 & \text{for } \psi \in \text{type 3} \\
			b^*_5 & \text{for } \psi \in \text{type 5} \\
			c^*   & \text{for } \psi \in \text{type 4}.
		\end{cases}\right.
	\end{equation*}
	It remains to prove $b^*_2=b^*_3=b^*_5$.
	Fortunately, the above simplification makes it much easier.
	
	Consider $\sigma_1=(\xii^t\xij^{2s} \leftrightarrow \xij^s\xii^t\xij^s\ |\ \forall i,j,t,s) \in \text{Sym}(\Psi)$. It is not true that in general $F(\sigma_1(x))=F(x)$. However, by using the previous simplification for $x^*$ and \fullref{tab:3}, it is easy to see that $F(\sigma_1(x^*))=F(x^*)$. Thus, again by \fullref{Uniquness}, we get
	\begin{equation*}
		b^*_2=x^*(\xi_1\xi_2^2)=x^*(\xi_2\xi_1\xi_2)=b^*_3.
	\end{equation*}
	Similarly, $\sigma_2=(\xi_1\xi_2^2 \leftrightarrow \xi_1\xi_2\xi_3) \in \text{Sym}(\Psi)$ gives us $b^*_2=b^*_5$ which completes the proof.
\end{proof}

Now that we have simplified $x^*$ as much as possible. We shall use the optimization question with the Karush-Kuhn-Tucker theorem to get further equalities on $x^*$ with a target of calculating exact values for its coordinates. 

\begin{theorem}
	For all $r_1, r_2 \in \F$, we have $f_{r_1}(x^*)=f_{r_2}(x^*)$.
	\label{OptimizationQuest}
\end{theorem}

\begin{proof}
	We will obtain the proof by applying the Karush-Kuhn-Tucker theorem to the optimization question
	\begin{equation*}
		\text{min } F(x) \text{ for } x \in \simp.
	\end{equation*}
	By \fullref{SymSimplification}, we can restrict our search for an optimal solution to those elements of $\simp$ showing the same coordinate distribution as $x^*$. In other words, we solve the problem below:
	\begin{align*}
		& \text{min } F(x) \text{ for } x\in \tilde{\simp} , \textnormal{ where }         \\
		& \tilde{\simp}:=\{ x\in \simp\ |\ \exists a,b,c\in (0,1),\
		x(\psi)=\left.
		\begin{cases}
			a & \text{for } \psi \in \text{type 1}                     \\
			b & \text{for } \psi \in \text{type 2 or type 3 or type 5} \\
			c & \text{for } \psi \in \text{type 4}.
		\end{cases}\right.
	\end{align*}
	A more illuminating observation comes from the fact that $X=1$ for all $x\in \tilde{\simp}$ and the counts in \fullref{tab:1}. For any $x\equiv(a,b,c)$, we have
	\begin{align}\label{aFormula}
		1 = 2na+8n(n-1)^2b+4n(n-1)c.
	\end{align}
	
	The displacement functions are simplified as it no longer matters what specific coordinate $\Psi_r$ contains, but the number of types of coordinates.
	As an example, for any type 1a relation $r$ and $x\equiv(a,b,c)\in\tilde{\simp}$, we have (with the help of \fullref{tab:3}) $x_r=a$ and $X_r=1-(a+4(n-1)^2b+2(n-1)c)$.
	Moreover using (\ref{aFormula}) to replace $a$, we obtain $X_r=(2n-1)/(2n)$.
	This simplifies $f_r(x)=(1-x_r)/{x_r}\cdot({1-X_r})/{X_r}$ into $f_r(x)=({1-a})/{a}\cdot {1}/{2n-1}$. We summarize these simplifications for all types of relations in $\F$ in \fullref{tab:5}.
	\begin{table}[h]
		\centering
		\raa{2}
		\begin{tabular}{llll}
			\toprule
			$r$                & $x_r$ & $X_r$             & $f_r(x)$                                        \\
			\midrule
			Type 1a            & $a$   & $\displaystyle\frac{2n-1}{2n}$ & $\displaystyle\frac{1-a}{a}\cdot \frac{1}{2n-1}$             \\
			\midrule
			Type 2b, 3a and 5a & $b$   & $1-(2(n-1)b+c)$   & $\displaystyle\frac{1-b}{b}\cdot\frac{1-2na}{4n(n-1)-1+2na}$ \\
			\midrule
			Type 4b            & $c$   & $\displaystyle\frac{1}{2n}$    & $\displaystyle\frac{1-c}{c}\cdot (2n-1)$                     \\
			\bottomrule
		\end{tabular}
		\caption{$f_r(x)$ for $x\equiv(a,b,c)\in\tilde{\simp}$}
		\label{tab:5}
	\end{table}
	
	By \fullref{Some4bAttains}, we know that there exists a Type 4b relation $r_0$ such that $\beta^*=f_{r_0}(x^*)\geq f_r(x^*)$ for all $r\in\F$.
	Applying this knowledge in light of \fullref{tab:5}, we get two new inequalities further narrowing the search for the optimal solution. We obtain
	\begin{gather*}
		\frac{1-a}{a}\cdot \frac{1}{2n-1} \leq \frac{1-c}{c}\cdot (2n-1)\quad\textnormal{ and }\quad
		\frac{1-b}{b}\cdot\frac{1-2na}{4n(n-1)-1+2na} \leq \frac{1-c}{c}\cdot (2n-1).
	\end{gather*}
	This simplifies the optimization question to its ultimate form as follows
	\begin{table}[H]
		\centering
		\raa{2}
		\begin{tabular}{ll}
			Target:      & $f(a,b,c)=\displaystyle\frac{(2n-1)(1-c)}{c}$ (minimize)                                     \\
			\midrule
			Constraints: & $g_1(a,b,c)=\displaystyle\frac{1-a}{a}\cdot\frac{c}{1-c}-(2n-1)^2\leq 0  $                   \\
			& $g_2(a,b,c)=\displaystyle\frac{1-2na}{4n(n-1)-1+2na}\cdot\frac{1-b}{b}\cdot\frac{c}{1-c}-(2n-1)\leq 0$ \\
			& $h(a,b,c)=\displaystyle 2na+8n(n-1)^2b+4n(n-1)c=1$                                            \\
			& $(a,b,c)\in U=(0,1)^3.$
		\end{tabular}
	\end{table}
	
	From hereon in, we shall apply the optimization techniques detailed in \cite[Chapter 7.2]{convex}.
	Our goal is to prove that the Lagrange multipliers for $g_1$ and $g_2$ can not be $0$, which (by Karush-Kuhn-Tucker Theorem as in \cite[Theorem 7.2.9]{convex}) will imply that the relevant inequalities in the constraints has to be active (equal) at the optimal point, proving our theorem.
	
	We shall first discuss the three prerequisites of Karush-Kuhn-Tucker Theorem. The existence of the optimal solution is proven in \fullref{Existence}.
	The target and constraint functions are obviously differentiable on $U$.
	We also need to show that the Mangasarian-Fromovitz constraint qualification \cite[Assumption 7.2.3]{convex} holds.
	The qualification for our question can be checked by verifying the statements 
	\begin{itemize}
	\item[(1)] if 
		the function $p \mapsto \nabla h(x^*)\cdot p$ from $\mathbb{R}^3$ to $\mathbb{R}$ is onto and, 
		\item[(2)]
		$\set{p\in\mathbb{R}^3}{\nabla h(x^*)\cdot p=0 \text{ and } \nabla g_i(x^*)\cdot p<0 \text{ for } i=1,2} \neq \emptyset.$
		\end{itemize}
	The first one is trivially true as $\nabla h$ is constant and nonzero.
	The second one also holds since the point $p=(2(n-1),0,-1)$ is in the set. Hence, by the Karush-Kuhn-Tucker theorem, there exists $\lambda_1,\lambda_2\in\mathbb{R}_+$ and $\lambda_3\in\mathbb{R}$ such that:
	\begin{align*}
		(\nabla f+\lambda_1\nabla g_1+\lambda_2\nabla g_2+\lambda_3\nabla h)(x^*\equiv (a^*,b^*,c^*))=\vec{0}
	\end{align*}
	This gives us the following linear system of $\lambda_1, \lambda_2$ and $\lambda_3$
	\begin{alignat}{4}
		-\frac{c^*}{(a^*)^2(1-c^*)}\lambda_1  & \hspace*{1.4cm} {}+{} \frac{8n^2(n-1)}{A}\lambda_2                  & {}+{} 2n\lambda_3        & {}={} 0,\label{lagrange:1}                    \\
		& \hspace*{0.8cm} {}-{} \frac{(1-2na^*)c^*}{A(b^*)^2(1-c^*)}\lambda_2 & {}+{} 8n(n-1)^2\lambda_3 & {}={} 0, \label{lagrange:2}                    \\
		\frac{1-a^*}{(a^*)(1-c^*)^2}\lambda_1 & {}+{} \frac{(1-2na^*)(1-b^*)}{Ab^*(1-c^*)^2}\lambda_2               & {}+{} 4n(n-1)\lambda_3   & {}={} \frac{2n-1}{(c^*)^2}, \label{lagrange:3}
	\end{alignat}
	where $A=4n(n-1)-(1-2na)$.
	It is easy to see that $1-2na\in(0,1)$ by using the constraint $h=1$. Thus, every parenthesis in this system is positive.
	This implies that the signs in front of the coefficients are the signs of the coefficients.
	
	If $\lambda_3=0$, then by (\ref{lagrange:2}) $\lambda_2=0$ as well,
	which implies $\lambda_1=0$ by (\ref{lagrange:1}).
	But, this is a contradiction to (\ref{lagrange:3}). Thus, $\lambda_3\neq 0$.
	The fact that $\lambda_3$ is non-zero implies $\lambda_2>0$ by (\ref{lagrange:2}),
	which in turn, again via (\ref{lagrange:2}), shows that $\lambda_3>0$ as well.
	Finally, the positiveness of $\lambda_2$ and $\lambda_3$ on (\ref{lagrange:1}) shows that $\lambda_1>0$. Therefore, the Lagrange multipliers of $g_1$ and $g_2$ are greater than $0$, proving what we wanted.
\end{proof}
\begin{theorem}\label{optimization:main}
	If $x^*\equiv (a^*,b^*,c^*)$ is as in \fullref{SymSimplification} and $\alpha$ is as defined after \fullref{Proots}, then we have
	\begin{itemize}
		\item[(i)] $\displaystyle a^*=\frac{1}{(2n-1)+\alpha}$,\quad
		$\displaystyle b^*=\frac{(2n-1)(\alpha-1)}{(4n^2-4n-1)(2n-1)\alpha^2+2(2n^2-1)\alpha-(2n-1)}$,\quad
		$\displaystyle c^*=\frac{2n-1}{2n-1+\alpha},$
		\item[(ii)] $\beta^*=\alpha^*=\alpha.$
	\end{itemize}
\end{theorem}
\begin{proof}
	By \fullref{OptimizationQuest}, we know that $f_r(x^*)=\beta^*$ for all $r\in \F$. By the proof of the same theorem, we also know that there are really three functions at the point of optimality.
	Therefore, we have the following equations
	\begin{align*}
		\frac{1-a^*}{a^*}\cdot \frac{1}{2n-1}                 & =\beta^*, \\
		\frac{1-b^*}{b^*}\cdot\frac{1-2na^*}{4n(n-1)-1+2na^*} & =\beta^*, \\
		\frac{1-c^*}{c^*}\cdot (2n-1)                         & =\beta^*.
	\end{align*}
	Simply solving this system for $a^*,b^*$ and $c^*$ yields that
	\begin{align*}
		a^* & =\frac{1}{(2n-1)+\beta^*,}                                                       \\
		b^* & =\frac{(2n-1)(\beta^*-1)}{(4n^2-4n-1)(2n-1)(\beta^*)^2+2(2n^2-1)\beta^*-(2n-1)}, \\
		c^* & =\frac{2n-1}{2n-1+\beta^*}.
	\end{align*}
	Putting these formulas into the equation (\ref{aFormula}) gives $\Pl (\beta^*)=0$, where $\Pl$ is the polynomial defined in \fullref{Proots}.
	But, by \fullref{AlphaIsBound}, we know that $1<\beta^*$. Since $\alpha$ is the only root of $\Pl$ greater than 1, we find $\beta^*=\alpha$.
	Applying that to  $\beta^*\leq\alpha^*\leq \alpha$ from \fullref{AlphaIsBound} gives us $\alpha=\beta^*\leq\alpha^*\leq\alpha$, also proving that $\alpha^*=\alpha$.
\end{proof}

%%%%%%%%%%%%%%%%%%%%%%%%%%%%%%%%%%%%% SECTION 6

\section{Proof of the Main Theorem}\label{sec:theend}

We will first prove two lemmas that will simplify the proof of the main theorem.

\begin{lemma}\label{theend:purelyloxoImpNonelementary}
	Let $g, h$ in $PSL(2,\C)$ and $G=\langle g, h \rangle$.
	If $G$ is freely generated by $\{g,h\}$, purely loxodromic and Kleinian, then $G$ is nonelementary.
\end{lemma}
\begin{proof}
	Assume $G$ is elementary.
	By the classification of discrete elementary groups done in \cite[Chapter 5.1]{Beardon}, there exists a two element subset of $\hat{\C}$ that is invariant under the action of $G$.
	Without loss of generality, take the invariant subset to be $\{0,\infty\}$.
	
	If the fixed points of $g$ were not $0$ and $\infty$, we would have $g(0)=\infty$ and $g(\infty)=0$.
	This implies that $g(z)={a}/{z}$, contradicting with the freeness of $G$ as $g^2=id$.
	Similarly, the fixed points of $h$ must be $0$ and $\infty$ as well. Hence, the fixed point of $g$ and $h$ are the same. By \cite[Theorem 4.3.5]{Beardon} we conclude that either $[g,h]=id$ or $[g,h]$ is parabolic.
	The first case contradicts with the freeness of $G$, and the second one contradicts the lemma hypothesis that $G$ is purely loxodromic.
\end{proof}

\begin{lemma}\label{theend:disp_midpoint_conj_ineq}
	Let $g, h \in PSL(2,\C)$ be noncommuting and loxodromic.
	If $z_0$ is the midpoint of the shortest geodesic segment connecting the axes of $g$ and $h^{-1}gh$, then $d_g z_0 < d_{hgh^{-1}} z_0$.
\end{lemma}
\begin{proof}
	For $\gamma \in PSL(2,\C)$, let $Z_\lambda(\gamma)=\set{z\in\mathbb{H}^3}{d_\gamma z \leq \lambda}$, the displacement cylinder of $\gamma$ with radius $\lambda$.
	By \cite[5.4.11]{Beardon}, $d_\gamma z$ is an increasing function of $\rho(z, A_\gamma)$.
	Thus, $Z_\lambda(\gamma)$ really is a cylinder with axis $A_\gamma$.
	Fix $\lambda=d_g z_0$ so that $z_0 \in Z_\lambda (g)$.
	By lemma hypothesis, we have $\rho(z_0,A_g)=\rho(z_0,A_{h^{-1}gh})$. Moreover, the terms $T_g$ and $\theta_g$ in \cite[5.4.11]{Beardon} are invariant under conjugation.
	Hence, $\rho(z_0,g z_0)=\rho(z_0, h^{-1}ghz_0)$ which proves that $z_0 \in Z_\lambda(h^{-1}gh)$.
	
	We claim that $z_0$ is the only element in $Z_\lambda(g) \cap Z_\lambda(h^{-1}gh)$.
	If it was not, then there would be an element $z \neq z_0$ such that $\rho(z,gz)=\rho(z,h^{-1}gh)$ which (again by \cite[5.4.11]{Beardon}) would imply that $\rho(z,A_g)=\rho(z,A_{h^{-1}gh})$.
	By lemma hypothesis, we would have $\rho(z_0,A_g) < \rho(z,A_g)$ so that $\lambda = \rho(z_0,gz_0) < \rho(z,gz)$, a contradiction.
	Now that we know $Z_\lambda(g) \cap Z_\lambda(h^{-1}gh)=\{z_0\}$, one can easily show that $Z_\lambda(g) \cap Z_\lambda(hgh^{-1})=\{hz_0\}$.
	Because $z_0 \in Z_\lambda(g)$, we can derive that $z_0 \notin Z_\lambda(hgh^{-1})$.
	Therefore, $\lambda = d_g z_0 < d_{hgh^{-1}} z_0$.
\end{proof}

The following theorem will be the principal effect of the optimization that was done in \fullref{sec:Optimization} to our main result.
The geometrically infinite case is quite straightforward due to the Culler-Shalen machinery.
However, the generalization to the geometrically finite case requires some high level theorems which we refer to inside the proof.

\begin{theorem}\label{theend:log_alpha}
	Let $\Gamma$ be a Kleinian, purely loxodromic subgroup of $PSL(2,C)$ that is freely generated by $\Xi=\{ \xi_1, \xi_2, \dots, \xi_n \}$.
	If $\Gamma_*=\{ 1, \xii^t, \xigent{i}{t}{j}{s}{i}{-t} \}$ and $\alpha$ is the unique root of the polynomial $\Pl$ greater than $(2n-1)^2$, then for any $z\in\mathbb{H}^3$ we have
	\begin{equation*}
		\max \set{d_\gamma z}{\gamma \in \Gamma_*} \geq \frac{1}{2}\log \alpha.
	\end{equation*}
\end{theorem}
\begin{proof}
	Fix $z \in \mathbb{H}^3$ and let $\lambda = \max \set{d_\gamma z}{\gamma \in \Gamma_*}$. Assume $\Gamma$ is geometrically infinite.
	By \fullref{CSmachine:main}($ii$), we have
	\begin{equation*}
		\lambda \geq d_{\gamma_r} z \geq \frac{1}{2} \log f_r(m)
	\end{equation*}
	for any relation  $r$ so that $\lambda \geq \frac{1}{2}\log G(m)$, where $G$ is defined in \fullref{sec:Tables}.
	Moreover, by \fullref{CSmachine:main}($i$) and \fullref{optimization:main}, we derive
	\begin{equation*}
		\lambda \geq \frac{1}{2}\log G(m) \geq \frac{1}{2}\log \left(\inf\set{G(x)}{x \in \simp}\right)=\frac{1}{2}\log\alpha.
	\end{equation*}

Assume that $\Gamma$ is geometrically finite. Our plan of attack is to prove that the infimum of all such lower boundaries over all possible free groups is attained at a special place close to the geometrically infinite case. We will have to use yet another optimization and some prior knowledge.
	Consider the function $f_z$ defined below
	\begin{align*}
		f_z: \psl^n & \longrightarrow \mathbb{R}_{+}                                \\
		\xi         & \longmapsto\text{max}\set{d_\gamma z}{\gamma\in\Gamma_*(\xi)}.
	\end{align*}
	Let $\gfFrak$ denote the subset of elements of $\psl^n$ whose components form a set of free generators of a purely loxodromic free subgroup of $\psl$ that is geometrically finite.
	As the image $f_z(\gfFrak)$ is bounded below and f is continuous, there exists a $\overline{\xi}\in\overline{\gfFrak}$ such that 
	\[f_z(\overline{\xi})=\text{inf}\set{f_z(\xi)}{\xi\in\gfFrak}.\]
	However, it has been shown in the proofs of \cite[Theorem 5.1]{Yu}, \cite[Theorem 4.1]{Yu2} and finally in \cite[Theorem 4.2]{Yu4} that any such optimal point must satisfy
	$\overline{\xi} \in \overline{\gfFrak}-\gfFrak.$
	Moreover, the subset of elements of $\psl^n$ whose components are generating purely loxodromic and geometrically infinite groups is dense in $\overline{\gfFrak}-\gfFrak$ as was shown in \cite[Propositions 8.2 and 9.3]{CSParadox}, \cite[Main Theorem]{CSH} and \cite[]{CCHS}.
	Hence, the geometrically finite case folds into the geometrically infinite case, completing the proof.
	\end{proof}

The final theorem of this paper is more general and in appearance more complicated than its counterpart in \cite{Yu4}. However, most of the complications in the proof have been handled by the previous lemmas and theorems in this paper and the arguments inside the proof is largely identical to that of the counterpart.

\begin{theorem}\label{thened:main_result}
	Let $\Gamma$ and $\alpha$ be as in \fullref{theend:log_alpha}.
	Assume that there exist $i_0, j_0$ such that the inequalities $d_{\xigent{i_0}{ }{j_0}{ }{i_0}{-1}}z_2 \leq d_{\xigent{i_0}{ }{j_0}{ }{i_0}{-1}}z_1$ and $d_\gamma z_2 < \frac{1}{2}\log \alpha$ for every $\gamma\in\Phi = \Gamma_* - \{\xi_{i_0},\ \xigent{j_0}{t}{i_0}{s}{j_0}{-t}\}$ hold, where $z_1$ and $z_2$ are the midpoints of the shortest geodesic segments connecting the axis of $\xi_{j_0}$ to the axes of $\xigent{i_0}{ }{j_0}{ }{i_0}{-1}$ and $\xigent{i_0}{-1}{j_0}{ }{i_0}{ }$, respectively. Then, we have
	\[\left|\tracesquare{\xi_{j_0}}-4\right|+\left|\trace{\xigent{j_0}{ }{i_0}{ }{j_0}{-1}\xi_{j_0}^{-1}}-2\right| \geq 2\sinh^2\left(\frac{1}{4}\log \alpha\right).\]
\end{theorem}

\begin{proof}
	We shall make use of the following facts. The first one is a straightforward exercises to the teachings of \cite{Beardon} and the second one is due to \fullref{theend:purelyloxoImpNonelementary}.
		The theorem remains invariant under conjugation.
				The subgroup $\langle \xi_{i_0},\ \xi_{j_0}\rangle$ must be non-elementary as it is free, Kleinian and purely loxodromic.

	Since the theorem remains invariant under conjugation, we may assume without loss of generality that $\xi_{i_0}$ fixes $0$ and $\infty$.
	This gives us some useful matrix presentations
	\begin{eqnarray*}
		\xi_{i_0}  = \begin{bmatrix} u & 0 \\ 0 & \displaystyle 1/u \end{bmatrix} \text{ for some } u = |u|e^{i\theta} & \textnormal{ and } &
		\xi_{j_0}  = \begin{bmatrix} a & b \\ c & d \end{bmatrix} \text{ with } ad-bc=1
		\end{eqnarray*}
	so that we can write
	\begin{equation*}
		|\tracesquare{\xi_{i_0}}-4|+|\trace{\xigent{j_0}{ }{i_0}{ }{j_0}{-1}\xi_{j_0}^{-1}}-2| =\left(u-\frac{1}{u}\right)^2\left(1+|bc|\right) = 4\sinh^2\frac{1}{2}T_{\xi_{i_0}}+4\sin(\theta),
	\end{equation*}
	where the last equation is due to \cite[equations 5.4.8 and 5.4.10]{Beardon}.
	
	By \cite[proof of Theorem 5.1.3 case(i)]{Beardon}, $\xi_{i_0}$ and $\xi_{j_0}\xi_{i_0}\xi_{j_0}^{-1}$ have no common fixed point. This implies $a,b,c,d \neq 0$. Consider the cross-ratio equation $[1,-1,w,-w] = [0,\infty,b/d,a/c]$ which is true for any $w$ that satisfies the following equation
	\begin{equation}\label{theend:bcw_equation}
		bc = \frac{(1-w)^2}{4w}.
	\end{equation}
	This equation has a unique solution for $w$ with norm greater than 1. Let $w = e^{2z_0}$ (for $z_0=x_0+iy_0 $) be that solution.
	Replacing $w$ in (\ref{theend:bcw_equation}) yields $bc = \sinh^2z_0$ which we can use to obtain the following chain of (in)equalities
	\begin{equation}
		4|bc|^2  = |\cosh^2z_0-1|^2    
		= (\cosh2x_0-\cos2y_0)^2 
		 \geq (\cosh2x_0-1)^2     
		 \geq (\cosh^2x_0-1)^2.
	\end{equation}
	Using the above inequality and  number of simplifications give
	\begin{align}\label{theend:1_plus_bc_ineq}
		1+|bc| \geq \frac{\cosh^2x_0}{2}.
	\end{align}
	
	Let $\mathcal{A}$ be the translation axis of $\xi_{i_0}$ and $\mathcal{B}$ be the translation axis of $\xigent{j_0}{}{i_0}{}{j_0}{-1}$. 
	Also notice that $\mathcal{B}=\xi_{j_0}\mathcal{A}$. 
	Since trace squared, translation length and the square of the sine square of translation angle are all invariant under conjugation, we have the following inequalities for all $z\in\mathbb{H}^3$
	\begin{eqnarray}\label{theend:long_ineq1}
		\sin^2\frac{1}{2}d_{\xi_{i_0}}z  & = &  \sinh\frac{1}{2}T_{\xi_{i_0}} \cosh d(z,\mathcal{A})+\sin^2\theta_{\xi_{i_0}}\sinh^2 d(z,\mathcal{A})\nonumber    \\                &   \leq & \left(\sinh^2\frac{1}{2}T_{\xi_{i_0}}+\sin^2\theta_{\xi_{i_0}}\right)\cosh^2 d(z,\mathcal{A})
		\end{eqnarray}
		and
\begin{eqnarray}		\label{theend:long_ineq2}
\sin^2\frac{1}{2}d_{\xigent{j_0}{}{i_0}{}{j_0}{^-1}}z & = &  \sinh\frac{1}{2}T_{\xi_{i_0}} \cosh d(z,\mathcal{B})+\sin^2\theta_{\xi_{i_0}}\sinh^2 d(z,\mathcal{B})\\ & \leq & \left(\sinh^2\frac{1}{2}T_{\xi_{i_0}}+\sin^2\theta_{\xi_{i_0}}\right)\cosh^2 d(z,\mathcal{B}).\nonumber
	\end{eqnarray}
	By using the above inequalities and the fact that the hyperbolic sine and cosine are increasing for $x\geq 0$, we conclude that
	\begin{align}\label{theend:themax_ineq}
		\sinh^2\frac{1}{2}\max\left\{d_{\xi_{i_0}}z,d_{\xigent{j_0}{}{i_0}{}{j_0}{^-1}}z\right\} \leq \frac{1}{4}\left|u-\frac{1}{u}\right|^2\cosh^2 \max\left\{d_z\mathcal{A},d_z\mathcal{B}\right\}.
	\end{align}
	Let $\Psi \in \mathcal{M}$ be defined by taking $[0,\infty,h0,h\infty]$ to $[1,-1,w,-w]$, where $w$ is defined in (\ref{theend:bcw_equation}).
	Under this setting $\Psi\mathcal{A}$ will be the geodesic from $-1$ to $1$ and $\Psi\mathcal{B}$ will be the geodesic from $-w$ to $w$ so that we have
	\begin{align*}
		d(\mathcal{A},\mathcal{B}) = d(j,|w|j) = \log |w| = 2x_0.
	\end{align*}
	By the above equality and the definition of $z_1$ given in theorem hypothesis, we can establish the fact that $d(z_1,\mathcal{A})=x_0=d(z_1,\mathcal{B})$. Applying this fact to the equations in (\ref{theend:long_ineq1}) and (\ref{theend:long_ineq2}) gives us another fact $d_{\xi_{i_0}}z_1 = d_{\xigent{j_0}{}{i_0}{}{j_0}{-1}}z_1$. These two facts can be used to simplify (\ref{theend:themax_ineq}) into the first of the following inequalities. Using (\ref{theend:1_plus_bc_ineq}) gives us the second one, which is
	\begin{align}\label{theend:finalineq}
		\sinh^2 \frac{1}{2} d_{\xi_{i_0}}z_1 \leq \frac{1}{4}\left|u-\frac{1}{u}\right|^2\cosh^2x_0 \leq \frac{1}{2}\left|u-\frac{1}{u}\right|^2(1+|bc|).
	\end{align}
	
	Finally, assume that the opposite of the theorem is true. In other words, assume that we have the inequality
	\begin{equation*}
		|\tracesquare{\xi_{j_0}}-4|+|\trace{\xigent{j_0}{ }{i_0}{ }{j_0}{-1}\xi_{j_0}^{-1}}-2| < 2\sinh^2\left(\frac{1}{4}\log \alpha\right).
	\end{equation*}
	Because we have $d_{\xigent{i_0}{ }{j_0}{ }{i_0}{-1}}z_2 \leq d_{\xigent{i_0}{ }{j_0}{ }{i_0}{-1}}z_1$ and $d_\gamma z_2 < \frac{1}{2}\log \alpha$ for every $\gamma\in\Phi$ by the hypothesis, we get $d_\gamma z_2 < \frac{1}{2}\log \alpha$ for all $\gamma \in \Gamma_*$ by (\ref{theend:finalineq}) and \fullref{theend:disp_midpoint_conj_ineq}.
	This is a contradiction to \fullref{theend:log_alpha}.	
\end{proof}

%%%%%%%%%%%%%%%%%%%%%%%%%%%%%%%%%%%%% REFERENCES

A. Nedim Narman, PhD

Department of Mathematics, Faculty of Engineering and Natural Sciences, 

İstanbul Bilgi University, İstanbul, Turkey. 

email: nedimnarman@gmail.com, \href{https://orcid.org/0000-0001-6819-9018}{https://orcid.org/0000-0001-6819-9018}\\

İlker S. Yüce, PhD

Department of Mathematics, Faculty of Arts and Sciences, 

Yeditepe University, İstanbul, Turkey. 

email: ilker.yuce@yeditepe.edu.tr, \href{https://orcid.org/0000-0003-3224-7625}{https://orcid.org/0000-0003-3224-7625}

%%%%%%%%%%%%%%%%%%%%%%%%%%%%%%%%%

\end{document}